\documentclass[smallextended,referee,envcountsect,]{svjour3}
\smartqed
\usepackage{graphicx}
\usepackage{subcaption}
\usepackage{amsfonts, amsmath, amssymb}
\usepackage{dsfont}
\usepackage{xcolor}
\usepackage{hyperref}
\usepackage{pifont}
\usepackage{soul}
\usepackage{subcaption}

\newcommand{\inprod}[2]{\left\langle #1, #2 \right\rangle}
\newcommand{\norm}[1]{\left\lVert #1 \right\rVert}
\newcommand{\abs}[1]{\left\lvert #1 \right\rvert}

\newcommand{\bb}[1]{\mathbb{#1}}
\newcommand{\dom}[1]{\mathrm{dom}{(#1)}}
\newcommand{\prox}[2]{\mathrm{Prox}_{#1}{\left(#2\right)}}
\newcommand{\proj}[2]{\mathrm{Proj}_{#1}{\left(#2\right)}}
\newcommand{\dd}{\mathrm{d}}

\DeclareMathOperator{\argmin}{argmin}

\newcommand{\xmark}{\ding{55}}%

\newtheorem{assumption}{Assumption}

\journalname{JOTA}

\title{An Inexact Halpern Iteration with Application to Distributionally Robust Optimization}

\subtitle{}

\author{Ling Liang $\cdot$ Zusen Xu $\cdot$ Kim-Chuan Toh $\cdot$ Jia-Jie Zhu}

\institute{Ling Liang, Corresponding author \at
             University of Maryland \\
              College Park, MD, USA 20742 \\
              liang.ling@u.nus.edu
              \and
              Zusen Xu \at 
              Weierstrass Institute for Applied Analysis and Stochastics \\
              Berlin, Germany 10117 \\
              xu@wias-berlin.de
            \and
              Kim-Chuan Toh  \at
              National University of Singapore \\
              Singapore 119076\\
              mattohkc@nus.edu.sg 
            \and 
            Jia-Jie Zhu \at 
            Weierstrass Institute for Applied Analysis and Stochastics \\
            Berlin, Germany 10117 \\
            zhu@wias-berlin.de
}

\date{Received: June 22, 2024 / Accepted: May 22, 2025}

\begin{document}

\maketitle

\begin{abstract}%
	The Halpern iteration for solving monotone inclusion problems has gained increasing interests in recent years due to its simple form and appealing convergence properties. In this paper, we investigate the inexact variants of the scheme in both deterministic and stochastic settings. We conduct extensive convergence analysis and show that by choosing the inexactness tolerances appropriately, the inexact schemes admit an $O(k^{-1})$ convergence rate in terms of the (expected) residue norm. Our results relax the state-of-the-art inexactness conditions employed in the literature while sharing the same competitive convergence properties. We then demonstrate how the proposed methods can be applied for solving two classes of data-driven Wasserstein distributionally robust optimization problems that admit convex-concave min-max optimization reformulations. We highlight its capability of performing inexact computations for distributionally robust learning with stochastic first-order methods and for general nonlinear convex-concave loss functions, which are competitive in the literature.
\end{abstract}


\noindent Communicated by Hailin Sun.

\section{Introduction}
\label{sec:introduction}

The monotone inclusion problem (also called the root-finding problem) is a mathematical problem that involves finding a solution/root satisfying an inclusion constraint, where the mapping involved is monotone. Mathematically, it has the following form 
\begin{equation}
    \label{eq-inclusion-G}
   \textrm{Find\;} z\in \bb{Z}\; \textrm{such that}\; 0\in G(z),
\end{equation}
where $\bb{Z}$ is a finite dimensional Euclidean space and $G:\bb{Z}\to2^\bb{Z}$ is a monotone set-valued mapping. Due to its wide applicability, the monotone inclusion problem \eqref{eq-inclusion-G} is of significant importance in a wide range of fields including optimization, variational analysis and machine learning \cite{rockafellar2009variational,facchinei2003finite,bauschke2011convex,rockafellar1976monotone,bottou2018optimization,zhang2022efficient}. As a representative example, in convex or convex-concave-min-max optimization problems, it is often necessary to find the root of a set-valued mapping that represents the (sub-)differential (which is monotone) of the objective function \cite{rockafellar2009variational}. Regarding algorithmic frameworks for solving \eqref{eq-inclusion-G}, fruitful numerical schemes have been proposed in the literature, which also result in many important applications. These existing algorithms include fixed-point iterative methods \cite{halpern1967fixed,iemoto2009approximating}, proximal point methods (or augmented Lagrangian methods) \cite{rockafellar1976monotone,luque1984asymptotic,cui2019r,liang2021inexact,yang2023corrected,zhao2010newton}, forward-backward splitting method  \cite{tseng2000modified,combettes2005signal,malitsky2020forward,cevher2021reflected,mainge2021fast}, operator splitting methods \cite{lions1979splitting,davis2017three}, (projected, optimistic or extra-) gradient methods \cite{cheney1959proximity,ceng2012extragradient,malitsky2015projected}, and their variants. We note that these algorithms are closely related to each other and may share similar convergence properties. 

In this paper, we focus on the case when $G$ is additionally assumed to be $\frac{1}{L}$-co-coercive with a constant $L>0$. Note that when $G$ is $\frac{1}{L} $-co-coercive, then $G$ is monotone (but it is not necessarily strongly monotone, consider e.g., constants mappings) and $L$-Lipschitz continuous. However, a monotone and Lipschitz continuous operator is not co-coercive generally. Consequently, co-coerciveness is an intermediate concept that lies between simple and strong monotonicity. Moreover, it is known that the gradient of a $L$-smooth convex function is $\frac{1}{L}$-co-coercive (see also in Lemma \ref{lemma:cvxcoco}). Consequently, the mapping $G$ naturally arises in a wide range of practical applications, especially in convex optimization problems with Lipschitz-smooth objective functions \cite{lan2020first}.

We first revisit the classic Halpern iteration, which dates back to \cite{halpern1967fixed}. For any given initial point $z^0\in \bb{Z}$, this method iteratively performs the following:
\begin{equation}
    \label{eq-halpern}
    z^{k+1} = \beta_k z^0 + (1-\beta_k)z^k - \eta_k G(z^k),\quad k\geq 0,
    \tag{\textrm{HI}}
\end{equation}
where $\beta_k\in (0,1)$ and $\eta_k > 0$ are parameters chosen suitably. One typical choice for these parameters would be $\beta_k = 1/(k+2)$ and $\eta_k:= (1 - \beta_k)/L$, for $k\geq 0$ \cite{tran2022connection}. Unlike classical algorithms, such as gradient-based methods that are specifically designed for solving optimization problems, Halpern iteration was originally developed to address fixed-point and inclusion problems, which encompass a broader class of optimization challenges. As a result, it can be readily extended to more general applications, including those discussed at the beginning of this paper. Additionally, while existing methods typically establish convergence rates in terms of the objective gap, Halpern iteration achieves an accelerated convergence rate with respect to the operator norm. These advantages have contributed to the growing interest in Halpern iteration in recent years. Though the Halpern iteration \eqref{eq-halpern} has a very simple form, it admits appealing convergence properties. The asymptotic convergence guarantees were established in \cite{wittmann1992approximation,leustean2007rates}. \cite{lieder2021convergence} first provided a direct proof of an $O(k^{-1})$ rate in terms of operator residue norm, i.e., $\norm{G(z^k)} = O(k^{-1})$. On the other hand, convergence analysis for \eqref{eq-halpern} and its variants based on potential functions has gained increasing attentions. For instance, \cite{diakonikolas2022potential} considered the potential function-based framework for minimizing gradient norms in convex and min-max optimization and argued that an $O(k^{-1})$ rate in terms of operator residue norm is perhaps the best one can hope for. \cite{yoon2021accelerated} considered a two-step variant of the Halpern iteration \eqref{eq-halpern} and applied it to min-max optimization problems. Recently, \cite{tran2022connection} revealed the connection between the Halpern iteration and the Nesterov's acceleration scheme \cite{nesterov1983method} which demonstrates deeper insights for both algorithms. 

However, existing works mostly assumed that $G$ can be evaluated exactly {though such an assumption holds only} in some special situations. In practical applications, such as the case when $G$ contains the resolvent of another operator, it is too restricted to require $G$ to have a closed-form expression. Particularly, when $G$ involves the projection onto a general convex set, such as a convex set taking the form of $\{z\in \bb{Z}\;:\; Az = b, Bz\leq d, z\in \cal{K}\}$ \footnote{Such a convex set arise from conic programming problems that are the most fundamental problems in optimization \cite{boyd2004convex}.}, where $A,B,c,d$ are given data and $\cal{K}$ is a convex cone, it can only be evaluated approximately by a certain iterative scheme. Therefore, it is natural to consider the inexact variant of the Halpern iteration for $k\geq 0$:
\begin{equation}
    \label{eq-inexact-halpern}
    \textrm{Find $\tilde{z}^k$ such that } \norm{G(z^k) - \tilde{z}^k}\leq \gamma_k,\; z^{k+1} = \beta_k z^0 + (1-\beta_k)z^k - \eta_k \tilde{z}^k,
    \tag{\textrm{iHI}}
\end{equation}
where $z^0\in \bb{Z}$ is the initial point and $\gamma_k \geq 0$ is the inexactness tolerance. However, to the best of our knowledge, the convergence properties of the above scheme have not been well-understood. The most related work considering this topic is perhaps \cite{diakonikolas2020halpern}. In particular, given a tolerance $\epsilon > 0$, \cite{diakonikolas2020halpern} considered replacing $G(z^k)$ with its approximation $\tilde{z}^k$ satisfying $\norm{G(z^k) - \tilde{z}^k} \leq \gamma_k := O \left( \epsilon / k^{2} \right)$ and established the $O(k^{-1})$ rate. However, the inexactness condition with an $O \left( \epsilon / k^{2} \right)$ tolerance can be 
{too stringent}, especially when $\epsilon$ is already small. Therefore, our first goal in the present paper is to conduct comprehensive convergence analysis that is suitable for a general co-coercive mapping $G$ with relaxed inexactness conditions for evaluating $G$. 

In practice, the mapping $G$ often involves a finite-sum structure. Therefore, stochastic variants of the Halpern fixed-point iterative scheme \eqref{eq-halpern} is necessary and is of independent interest. We mention some of the recent works that are related to this topic. Specifically, stochastic gradient methods for solving finite-sum quasi-strongly convex problems were analyzed in \cite{gower2019sgd,loizou2021stochastic} under the expected smoothness and co-coerciveness conditions. In \cite{cai2022stochastic}, inexact evaluation of $G$ in the Halpern iteration under a stochastic setting was considered, i.e., $G$ is estimated via random sampling and variance reduction through stochastic access. \cite{tran2023accelerated} considered two randomized block-coordinate optimistic gradient methods that are highly related to the Halpern iterative scheme \eqref{eq-halpern} for finite-sum maximal monotone inclusion problems. To the best of our knowledge, \cite{cai2022stochastic} might be the only work that considered the stochastic variants of \eqref{eq-halpern} directly. However, when $G$ contains the resolvent operator of another operator (see Section \ref{sec:finitesum} for more details) that can not be evaluated exactly, the results in \cite{cai2022stochastic} may not be applicable directly. Hence, combining theoretical guarantees under both deterministic and stochastic settings is needed to resolve this issue,  which is our second goal in this paper \footnote{A similar insight appeared in a different context in the independent work \cite{alacaoglurevisiting}, whose arXiv preprint \cite{alacaoglu2024extending} was released at the same time as ours.}. 

The capability of the inexact evaluation of the concerned mapping in the Halpern iteration further inspires us to investigate potential applicability of the method and its variants for solving modern data-driven optimization problems. In a traditional data-driven optimization problem, one makes a decision that generally performs well on a given training data set, denoted by $\{\hat \xi_i\}_{i = 1}^N \subseteq \Xi $, via solving the following empirical risk minimization (ERM) problem \cite{vapnik1991principles}:
\begin{equation}
\label{eq-erm}
	\min_{x \in \mathcal{X}}\; \bb{E}_{\xi\sim \hat{\mu}_N} [\ell(x,\xi)],
\end{equation}
where $\Xi \subseteq \bb{R}^d $ is the support set and $\mathcal{X}\subseteq \bb{R}^n$ is the feasible set that are both nonempty, $\hat{\mu}_N:= \frac{1}{N}\sum_{i = 1}^N \delta_{\hat{\xi}_i}$ denotes the empirical distribution with $\delta_\xi$ representing the Dirac distribution at $\xi\in\bb{R}^d$, and $\ell:\bb{R}^n\times \bb{R}^d\to \overline{\bb{R}}$ is the loss function with $\overline{\bb{R}}:=\bb{R}\cup\{\pm\infty\}$ denoting the extended reals. However, a decision performing well on the training data set may not perform well on unseen testing data sets, which is partially due to the discrepancy between the empirical distribution and the true (unknown) distribution governing the random variable $\xi\in \Xi$. For example, the ERM model is reported to be vulnerable to adversarial attacks such as adding small perturbations to the training data \cite{szegedy2013intriguing}. One popular way to improve robustness in decision-making processes is to include more informative distributions in \eqref{eq-erm} so that the obtained decision perform well for all data sets governed by these distributions. In other words, instead of making decisions via minimizing the expected loss $ \bb{E}_{\xi\sim \hat{\mu}_N} [\ell(x,\xi)] $, it is often preferable to consider minimizing the worst-case expected loss over a set of distributions near the empirical distribution $\hat{\mu}_N$, namely the ambiguity set. This approach is known as distributionally robust optimization (DRO) \cite{delageDistributionallyRobustOptimization2010}; see \cite{rahimian2019distributionally,rahimian2022frameworks,blanchet2024distributionally} for recent surveys of DRO. To the best of our knowledge, the Halpern iteration has not been applied for solving DRO in the literature. Thus, the third goal of this paper is to investigate the possibility of applying the Halpern iteration and its inexact variants for solving some important data-driven DRO problems, such as the data-driven Wasserstein DRO (WDRO) problems; see Section \ref{sec:WDRO} for more details.

\begin{table}[!]
    \centering
    \begin{tabular}{|c|c|c|c|l|} \hline 
         Work & Problem & Inexactness  & Convergence \\ \hline
         \cite{halpern1967fixed,wittmann1992approximation,leustean2007rates} & $z = G(z)$ & \xmark  & z$^k\to z^*$  \\ \hline
         \cite{lieder2021convergence} & $z = G(z)$ & \xmark &  $\|z^k-G(z^k)\| = O(\frac{1}{k})$  \\ \hline
         \cite{yoon2021accelerated} & $\displaystyle\min_x\max_yL(x,y)$ &   \xmark & $\|\nabla L(x^k,y^k)\|^2 = O(\frac{1}{k^2})$  \\ \hline
         \cite{tran2022connection} & $0\in G(z)$ & \xmark & $\|G(z^k)\|^2 = O(\frac{1}{k^2})$ \\ \hline 
         \cite{diakonikolas2020halpern,cai2022stochastic} & $0\in G(z)$ &  $O (\frac{\epsilon}{k^2})$ & $\|G(z^k)\|,\mathbb{E}[\|G(z^k)\|]= O(\frac{1}{k}) + O(\epsilon)$ \\ \hline 
         ours & $0\in G(z)$ & $O(\frac{\epsilon}{\sqrt{k}})$ &  $\|G(z^k)\|,\mathbb{E}[\|G(z^k)\|]= O(\frac{1}{k}) + O(\epsilon)$ \\ 
          & & $\widetilde{O}(\frac{1}{k^{3/2}})$ & $\|G(z^k)\|,\mathbb{E}[\|G(z^k)\|]= O(\frac{1}{k})$ \\ \hline 
    \end{tabular}
    \caption{Halpern iterations. $L$ is a convex-concave loss, $\epsilon\geq 0$ is a given tolerance and $k\geq 1$ denotes the iteration count. Inexactness stands for the inexact condition used for approximating $G$ and a \xmark-mark means that $G$ must be evaluated exactly.}
    \label{tab:halpern}
\end{table}

The contributions of this paper can be summarized as follows. We extensively analyze the inexact variants of the classical Halpern fixed-point iterative scheme in both deterministic and stochastic scenarios. We show, in particular, that by selecting appropriate inexactness tolerances, the $ O(k^{-1}) $ convergence rate of the residue norm, i.e., $\norm{G(z^k)}$, remains true. We also provide the same convergence rate for $\norm{z^{k+1} - z^k}$ as a byproduct. When the mapping $G$ admits a finite-sum structure, by using an efficient ProbAbilistic Gradient Estimator (PAGE) as in \cite{cai2022stochastic,li2021page}, we show that the stochastic Halpern iteration has the same $O(k^{-1})$ convergence rate in expectation, under some mild conditions. We emphasize that our inexactness conditions relax the current state-of-the-art conditions used in the literature; See Table \ref{tab:halpern} for a comparison between several Halpern-type methods. Moreover, we demonstrate how one can apply the Halpern iteration and its variants for solving two important data-driven WDRO problems. In particular, we demonstrate that the proposed techniques can be applied to the associated monotone inclusion problem obtained from their convex-concave min-max optimization reformulations for the Wasserstein distributionally robust {learning with generalized linear models} and the WDRO problems with convex-concave loss functions. The {adaptation} of the Halpern iterations for solving WDRO problems is novel as it allows one {to evaluate} the underlying mapping (e.g., the projection mapping onto a closed convex set) inexactly, which makes the proposed methods attractive. Moreover, given the deep connection between the Halpern iteration and the Nesterov's acceleration scheme \cite{tran2022connection}, our results could offer more insights into the inexact and stochastic variants of Nesterov's acceleration scheme, thus inspiring their applicability in solving DRO problems. Importantly, many existing WDRO algorithms are specifically designed for small-to-medium-scale problems with special loss functions that admit exact reformulations. In contrast, we show an inexact stochastic Halpern iteration that can solve general nonlinear convex-concave WDRO problems.

The rest of this paper is organized as follows. We first provide some notation and definitions that are necessary for our later exposition in Section \ref{sec:preliminary}. We next analyze the inexact variants of the Halpern iteration in both deterministic and stochastic settings in Section \ref{sec:iHalpern} and Section \ref{sec:finitesum}, respectively. Then, in Section \ref{sec:WDRO}, we show how the presented methods can be useful for solving two important classes of WDRO problems. To validate the practical performance of the proposed framework, we conduct some preliminary numerical experiments in Section \ref{sec:numerical}. Finally, we summarize the paper in Section \ref{sec:conclusions}. We defer all the technical lemmas and necessary proofs to the appendix. 

\section{Notation and Definitions}
\label{sec:preliminary}

We use $\bb{X}$,  $\bb{Y}$ and $\bb{Z}$ to denote finite dimensional Euclidean spaces equipped with the standard inner product $\inprod{\cdot}{\cdot}$ and the induced Euclidean norm $\norm{\cdot}$. Moreover, we use $\|\cdot\|_\infty$ to denote the infinity norm. Let $n$ be a positive integer, the standard basis for $\bb{R}^n$ is denoted as $\{e_1,\dots, e_n\}$ and the vector of all ones in $\bb{R}^n$ is denoted as $\mathds{1}_n$. We use $\overline{\bb{R}}$ to represent the extended reals $\bb{R}\cup\{\pm\infty\}$. Let $f:\bb{X}\to\bb{R}$ be a differentiable function, its gradient at a point $x\in \bb{X}$ is denoted as $\nabla f(x)$. If $\bb{X} = \bb{R}$, i.e., $f$ is a univariate function, then its derivative is denoted as $f'$. If $f:\bb{X}\times\bb{Y}\to\bb{R}$ is a differentiable {function} of two vector variables, we use $\nabla_xf(x,y)$ and $\nabla_yf(x,y)$ to denote the partial derivatives with respect to $x\in \bb{X}$ and $y\in \bb{Y}$, respectively. In this way, the gradient of $f$ at the point $(x,y)$ is denoted as $\nabla f(x,y) = [\nabla_xf(x,y);\nabla_yf(x,y)]$, where $[\cdot;\dots;\cdot]$ is the column vector stacked from its components.

\textbf{Convex functions.} The effective domain of a function $f:\bb{X}\rightarrow \overline{\bb{R}}$ is denoted as $\dom{f}$ on which $f$ takes finite values and we say that $f$ is convex if, for any $x_1,x_2\in \dom{f}$ and $\alpha\in [0,1]$, it holds that $f(\alpha x_1 + (1-\alpha)x_2) \leq \alpha f(x_1) + (1-\alpha)f(x_2)$. Note that $f$ is concave if $-f$ is convex. If $f$ is not identically $+\infty$, we say that  $f$ is a proper function. Moreover, $f$ is closed if it is lower semi-continuous. For any $x\in \dom{f}$, the sub-gradient of $f$ at $x$ is defined as $\partial f(x) := \{ v \in \bb{X}: f(x')\geq f(x) + \inprod{v}{x'-x},\; \forall x\in \bb{X} \}$. For a proper, closed and convex function, the proximal mapping of $f$ at $x$ is defined as $\mathrm{Prox}_f(x):= \argmin_{x'\in\bb{X}}\left\{ f(x') + \frac{1}{2}\norm{x' - x}^2\right\}$. Let $\cal{C}$ be a nonempty closed convex set, we use $\cal{I}_{\cal{C}}(\cdot)$ to denote the indicator function for $\cal{C}$. Then, it is clear that $\cal{I}_{\cal{C}}(\cdot)$ is a closed proper convex function. In this case, we use $\proj{\cal{C}}{\cdot} \equiv \mathrm{Prox}_{\cal{I}_{\cal{C}}}(\cdot)$ to denote the projection onto the set $\cal{C}$.

\textbf{Set-valued mappings.} Let $G:\bb{X}\rightrightarrows 2^\bb{X}$ be a set-valued mapping, we use $\dom{G}:=\left\{x\in \bb{X}:G(x)\neq \emptyset\right\}$ to denote its domain and use $\mathrm{graph}(G):=\left\{(x,x')\in \bb{X}\times \bb{X}: x'\in G(x)\right\}$ to denote its graph. Then the inverse of $G$ is denoted as $G^{-1}(x'):=\left\{x\in\bb{X}: x'\in G(x)\right\}$. We say that $G$ is single-valued if $G(x)$ is a singleton for any $x\in \bb{X}$.

\textbf{Monotonicity.} A set-valued mapping $G:\bb{X}\rightrightarrows 2^\bb{X}$ is said to be monotone if 
\[
    \left\langle x_1'  - x_2', x_1 - x_2\right\rangle \geq 0,\quad \forall x_1,\;x_2\in \dom{G},\; x_1'\in G(x_1),\; x_2'\in G(x_2).
\]
If there exists $\alpha > 0$ such that 
\[
    \left\langle x_1'  - x_2', x_1 - x_2\right\rangle \geq \alpha \left\lVert x_1 - x_2\right\rVert^2,
\]
for $\forall x_1,\;x_2\in \dom{G},\; x_1'\in G(x_1),\; x_2'\in G(x_2)$, then we say that $G$ is $\alpha$-strongly monotone. Let $G$ be a monotone mapping, if the graph of $G$ is not properly contained in the graph of any other monotone operator, then $G$ is said to be maximally monotone. Note that for a closed proper convex function $f:\bb{X}\to\overline{\bb{R}}$, its sub-gradient $\partial f$ is maximally monotone. 

\textbf{Resolvent operators.} For a monotone operator $G:\bb{X}\rightrightarrows 2^\bb{X}$, we denote its resolvent operator as $J_G:=(I+G)^{-1}$, where $I$ denotes the identity mapping. It is easy to verify that when $G$ is monotone, $J_G$ is single-valued with $\dom{G} = \bb{X}$. Note also that for a closed, proper and convex function $f:\bb{X}\to \overline{\bb{R}}$ and any positive scalar $\alpha\in \bb{R}$, it holds that $J_{\alpha \partial f}(x) = \prox{\alpha f}{x}$. In particular, for any closed convex set $\cal{C}\subseteq \bb{X}$ and any positive scalar $\alpha\in \bb{R}$, it holds that $J_{\alpha\partial\cal{I}_{\cal{C}}}(\cdot) = \proj{\cal{C}}{\cdot}$, i.e., the projection operator onto $\cal{C}$.

\textbf{Lipschitz continuity.} Let $f:\bb{X}\to\overline{\bb{R}}$ be a function, if there exists $L\geq 0$ such that 
\[
    \left\lvert f(x_1) - f(x_2)\right\rvert \leq L\left\lVert x_1 - x_2\right\rVert,\quad \forall x_1,x_2\in \dom{f},
\]
then $f$ is said to be $L$-Lipschitz continuous. A differentiable function $f$ is said to be $L$-smooth with modulus $L\geq 0$ if 
\[
    \left\lVert \nabla f(x_1) - \nabla f(x_2)\right\rVert \leq L\left\lVert x_1 - x_2\right\rVert,\quad \forall x_1,x_2\in \dom{f}.
\]
Let $G:\bb{X}\rightrightarrows 2^\bb{X}$ be a single-valued mapping, if there exists $L\geq 0$ such that
\[
    \left\lVert G(x_1) - G(x_2)\right\rVert \leq L\left\lVert x_1 - x_2\right\rVert,\quad \forall x_1,x_2\in \dom{G},
\]
then $G$ is said to be $L$-Lipschitz continuous. 

\textbf{Co-coerciveness.} Let $L > 0$. A set-valued mapping $G:\bb{X}\rightrightarrows 2^\bb{X}$ is said to be $\frac{1}{L}$-co-coercive if the following relation holds:
\[
    \left\langle G(x_1) - G(x_2), x_1 - x_2\right\rangle \geq \frac{1}{L}\left\lVert G(x_1) - G(x_2)\right\rVert^2,\quad \forall x_1,x_2\in \dom{G}.
\]
In the particular case when $L = 1$, we also say that $G$ is firmly non-expansive. Moreover, if $G$ is monotone, then its resolvent operator $J_G$ is firmly non-expansive. It is well-known that for any proper, closed and convex function $f:\bb{X}\to\overline{\bb{R}}$, its proximal mapping is firmly non-expansive, i.e.,
\begin{equation*}
    \inprod{\prox{f}{x_1} - \prox{f}{x_2}}{x_1 - x_2} \geq \norm{\prox{f}{x_1} - \prox{f}{x_2}}^2,\; \forall x_1,\;x_2\in \bb{X}. 
\end{equation*}
The following lemma will be useful in our later analysis.
\begin{lemma}[{\cite[Theorem 18.15]{bauschke2011convex}}]
\label{lemma:cvxcoco}
The gradient $\nabla f$ of a convex and $L$-smooth function $f:\bb{X}\to \overline{\bb{R}}$ is $\frac{1}{L}$-co-coercive.
\end{lemma}

Let $E:\bb{X}\rightrightarrows 2^\bb{X}$ and $F:\bb{X}\rightrightarrows 2^\bb{X}$ be two maximal monotone operators. Consider the set-valued mapping 
\[
    G(x):= \frac{1}{\alpha}\left( x - J_{\alpha E}\left( x - \alpha F(x)\right) \right) ,\quad \forall\; x\in \bb{X}.
\]
The following lemma shows that $G(x):=\frac{1}{\alpha}\left(x - J_{\alpha E}(x - \alpha F(x))\right)$  is co-coercive if $F$ is co-coercive under appropriate condition on $\alpha$; see, e.g., \cite[Proposition 26.1]{bauschke2011convex} for a proof of the lemma.
\begin{lemma}
    \label{lemma:G-coco}
    Suppose that $F$ is $\frac{1}{L}$-co-coercive, and $0<\alpha < \frac{4}{L}$. Then, $G$ is $\frac{\alpha(4-\alpha L)}{4}$-co-coercive.
\end{lemma}

We refer the reader to \cite{rockafellar1997convex,bauschke2011convex} for more advanced topics on convex analysis and monotone operator theory.

\section{An inexact Halpern fixed-point iterative scheme}
\label{sec:iHalpern}
As mentioned in the introduction, in the ideal case when $G(z)$ can be evaluated exactly for any $z\in \bb{Z}$, the exact Halpern iterative scheme \eqref{eq-halpern} is guaranteed to possess an $O(k^{-1})$ rate. In this section, we illustrate that when the inexactness tolerance $\gamma_k$ satisfies some mild conditions, the inexact version \eqref{eq-inexact-halpern} shares a similar convergence rate. We assume for the rest of this paper that the problem \eqref{eq-inclusion-G} admits at least one solution.

The main convergence properties of the inexact Halpern iteration \eqref{eq-inexact-halpern} are presented in the following theorem, whose proof is provided in Appendix \ref{proof-thm-rate}.

\begin{theorem}
	\label{thm-convergence-rate}
	Suppose that $G$ is $\frac{1}{L}$-co-coercive for a constant $L>0$. Let $\{z^k\}$ be generated by \eqref{eq-inexact-halpern} with $\beta_k = 1/(k+2)$ and $\eta_k = (1-\beta_k)/L$, then it holds that
	\begin{align}
		\norm{G(z^k)}^2 \leq &\;  \frac{\left(7L\norm{z^0 - z^*} + 10\sqrt{\sum_{i = 0}^{k-1}(i+1)^2\gamma_i^2}\right)^2}{(k+1)(k+2)},\quad k\geq 0, \label{thm-rate-1}\\
		\norm{z^{k+1} - z^k}^2 \leq &\; \frac{8\left(7L\norm{z^0 - z^*} + 11\sqrt{\sum_{i = 0}^{k}(i+1)^2\gamma_i^2}\right)^2 }{L^2(k+1)(k+2)}, \quad k\geq 0,\label{thm-rate-2}
	\end{align}
	where $z^* \in \bb{Z}$ is any solution such that $G(z^*) = 0$.
\end{theorem}

As we can see in the {above} theorem, the choice of the tolerance $\gamma_k$ would affect the rate of convergence of the proposed inexact Halpern  iteration \eqref{eq-inexact-halpern}. In particular, we have the following corollary, which relaxes the inexactness condition proposed in \cite{diakonikolas2020halpern}. The results in the corollary follow from Theorem \ref{thm-convergence-rate} directly, and we omit the proof. 

\begin{corollary}
	\label{corollary-epsilon}
	Let $ \epsilon \in (0,1] $ be given and assume that $G$ is $\frac{1}{L}$-co-coercive for a constant $L>0$. Suppose that $\{z^k\}$ is the sequence generated by \eqref{eq-inexact-halpern} with $\beta_k = 1/(k+2)$, $\eta_k = (1-\beta_k)/L$, and $\gamma_k:= \epsilon/\sqrt{k+1}$ for $k\geq 0$. Then, it holds that
	\[
		\norm{G(z^k)} \leq O(k^{-1}) + O(\epsilon),\quad \norm{z^{k+1} - z^k}\leq O(k^{-1}) + O(\epsilon),\quad  k\geq 0.
	\]
\end{corollary}

To the best of our knowledge, existing results for the convergence rate of the Halpern iteration are mainly stated in a similar form of Corollary \ref{corollary-epsilon}. However, in the case when $\epsilon$ is small, i.e., more accurate approximate solution is needed, the inexact condition with $\gamma_k:=O(\epsilon/\sqrt{k})$ requires evaluating $G$ with a relatively high accuracy even in the early stages of algorithm. We argue that this may be a disadvantage in a practical implementation since it may require more computational efforts. To alleviate this issue, the next corollary (whose proof is also omitted for simplicity) shows that it is possible to start with a relatively large tolerance and progressively decrease its value. But the speed of decrease for this choice of $\gamma_k$ is faster than that of Corollary \ref{corollary-epsilon}.

\begin{corollary}
	\label{corollary-large}
	Suppose that the sequence $ \{(k+1)^2\gamma_k^2\} $ is summable and assume that $G$ is $\frac{1}{L}$-co-coercive for a constant $L>0$. Let $\{z^k\}$ be the sequence generated by \eqref{eq-inexact-halpern} with $\beta_k = 1/(k+2)$, and $\eta_k = (1-\beta_k)/L$ for $k\geq 0$. Then, it holds that
	\[
		\norm{G(z^k)} \leq O(k^{-1}),\quad \norm{z^{k+1} - z^k}\leq O(k^{-1}),\quad k\geq 0.
	\]
\end{corollary}

Note that the condition that $ \{(k+1)^2\gamma_k^2\} $ is summable is quite mild.  For example, we can choose $\gamma_k:= (k+1)^{-a}$ with $a > 3/2$, for $k\geq 0$. Then, when $k+1 < \epsilon^{-1/(a - 1/2)}$, it holds that ${\gamma_k=}(k+1)^{-a} > \epsilon/\sqrt{k+1}$. 
Hence, $\gamma_k$ is less stringent than the tolerance required
in Corollary \ref{corollary-epsilon}. Moreover, after $K:=O(\epsilon^{-1/(a - 1/2)})$ iterations, the proposed algorithm is already able to produce an approximate solution $z^K$ such that $\norm{G(z^K)} \leq O(\epsilon^{1/(a - 1/2)})$, which may be good enough for termination. Nevertheless, we have shown that by choosing $\gamma_k$ appropriately, the inexact Halpern iteration \eqref{eq-inexact-halpern} computes an  $O(\epsilon)$-optimal solution, i.e., $\norm{G(z^k)}\leq O(\epsilon)$ within $O(\epsilon^{-1})$ iterations. 

\section{A stochastic Halpern iteration with PAGE}
\label{sec:finitesum}
In many modern applications in operation research and machine learning, $G$ usually admits a finite-sum structure. To see this, let $E:\bb{Z}\to 2^\bb{Z}$ and $F_i:\bb{Z}\to 2^\bb{Z}$, for $i = 1,\dots, N$, be maximal monotone operators. In addition, $F_i$ is assumed to be $\frac{1}{L_0}$-co-coercive for any $i = 1,\dots, N$. Then, many real-world applications can be reformulated as the following inclusion problem: 
\begin{equation}
	0\in E(z) + F(z):= E(z) + \frac{1}{N}\sum_{i = 1}^N F_i(z),\quad z\in \bb{Z},
	\label{eq-finitesum}	
\end{equation}
which is equivalent to the following new problem involving the resolvent operator of $E$:
\begin{equation}
	\label{eq-finitesum-G}
	0\in G(z):= \frac{1}{\alpha}\left(z - J_{\alpha E}\left(z -  \frac{\alpha}{N}\sum_{i = 1}^NF_i(z)\right)\right),\quad z\in \bb{Z}.
\end{equation}
Then, from Lemma \ref{lemma:G-coco}, we see that $G$ is $ \frac{\alpha(4-\alpha L_0)}{4} $-co-coercive provided that $\alpha \in (0, 4/L_0)$. Therefore, the results established in Theorem \ref{thm-convergence-rate} can be applied directly to guarantee
the $O(k^{-1})$ convergence rate in terms of the {residue} norm of $G$ with suitably chosen inexact tolerance $\gamma_k$. 

However, when $N$ is large, which is typically the case in practical applications, evaluating the sum $ \frac{1}{N}\sum_{i = 1}^NF_i(z) $ can be costly. This motivates us to consider stochastic/randomized variants of the inexact Halpern fixed-point iterative scheme \eqref{eq-inexact-halpern}. Specifically, given $z^0\in \bb{Z}$, it is natural to consider performing the following steps for $k\geq 0$:
\begin{equation}
\label{eq-isHalpern}
\left\{
	\begin{aligned}
		&\; \textrm{Find a stochastic estimator $ \widetilde{F}(z^k) $ satisfying $ \bb{E}\left[\norm{F(z^k) - \widetilde{F}(z^k)}^2\right] \leq \frac{1}{2} \sigma_k^2 $}, \\ 
		&\; 	\textrm{Find $\bar{z}^k$ such that } \norm{J_{\alpha E}\left(z^k - \alpha \widetilde{F}(z^k) \right) - \bar{z}^k}\leq \frac{\sqrt{2}}{2}\alpha \gamma_k,\\
		&\; \tilde{z}^k = \frac{1}{\alpha}(z^k - \bar{z}^k), \quad  z^{k+1} = \beta_k z^0 + (1-\beta_k)z^k - \eta_k \tilde{z}^k, 
	\end{aligned}
	\right.
    \tag{\textrm{isHI}}
\end{equation}
where $ \alpha\in (0,4/L_0)$, $\beta_k = 1/(k+2)$, $\eta_k  = (1-\beta_k)/L$ with $L:= 4/(\alpha(4-\alpha L_0))$, $\gamma_k \geq 0$ and $\sigma_k\geq 0$ are chosen appropriately. Here, we use $\bb{E}$ to denote the expectation with respect to all the randomness at any iteration of \eqref{eq-isHalpern}.

\begin{remark}
\label{remark-stochastic-halpern}	
We remark here that \cite{cai2022stochastic} considered only the case when $\gamma_k = 0$ for all $k\geq 0$ with $J_{\alpha E}$ denoting the projection onto a certain closed convex set. Hence, when $J_{\alpha E}$ can not be evaluated exactly or $J_{\alpha E}$ is the resolvent mapping for a general mapping $E$, the results in \cite{cai2022stochastic} can not be directly applicable.
\end{remark}

Let us assume temporally the existence of the desired stochastic estimator $\widetilde{F}$. Then, we can state our expected convergence rates in the following theorem. See Appendix \ref{proof-thm-rate-ishi} for a proof.
\begin{theorem}
	\label{thm-rate-ishi}
	Suppose that $G$ is $\frac{1}{L}$-co-coercive for a constant $L>0$. Let $\{z^k\}$ be the sequence generated by \eqref{eq-isHalpern} for solving \eqref{eq-finitesum-G}. Then it holds that 
	\begin{align}
		\bb{E}\left[\norm{G(z^k)}^2\right] \leq &\;  \frac{\left(7L\norm{z^0 - z^*} + 10\sqrt{\sum_{i = 0}^{k-1}(i+1)^2(\sigma_i + \gamma_i)^2}\right)^2}{(k+1)(k+2)},\quad k\geq 0, \label{eq-thm-rate-ishi-G}	\\
		\bb{E}\left[\norm{z^{k+1} - z^k}^2\right] \leq &\; \frac{8\left(7L\norm{z^0 - z^*} + 11\sqrt{\sum_{i = 0}^{k}(i+1)^2(\sigma_i + \gamma_i)^2}\right)^2 }{L^2(k+1)(k+2)}, \quad k\geq 0,\label{eq-thm-rate-ishi-Z}
	\end{align}
	where $z^*$ is a solution such that $G(z^*) = 0$.
\end{theorem}

Now, to guarantee that the scheme \eqref{eq-isHalpern} is well-defined, we need to ensure the existence of a stochastic estimator $\widetilde{F}$. If we let $\widetilde{F}(z^k) = F(z^k)$,  then the condition in the first line of \eqref{eq-isHalpern} trivially holds. On the other hand, the variance reduction techniques provide efficient ways to derive stochastic gradient estimators that admit {desirable} properties, such as a lower sample complexity than {that of} the vanilla stochastic gradient estimator. Motivated by the recent interest in developing variance reduced methods for solving general optimization problems, in this paper, we consider the popular ProbAbilistic Gradient Estimator (PAGE) \cite{li2021page,cai2022stochastic}. In particular, we have the following result that recursively quantifies the variance of the PAGE variance-reduced estimator, under standard conditions. To present the result, we need the following condition on the boundedness of the variance of the random query of $F$. The assumption of unbiased samples with bounded variance is widely used in the literature for analyzing the convergence properties of stochastic algorithms, such as stochastic gradient descent and its variants. This assumption helps control the inherent randomness introduced by stochastic gradient estimators; see, for example, \cite{ghadimi2016mini,lan2020first,li2021page,cai2022stochastic}. Without bounded variance, the updates could become excessively erratic, causing the iterates to oscillate unpredictably and hindering convergence to the optimal solution.
\begin{assumption}
    \label{assumption-bounded-var}
    There exists a positive constant $\sigma$ such that for any random index $i$ drawn from a given distribution, it holds that $\bb{E}\left[F_i(z)\right] = F(z)$ and $\bb{E}\left[\norm{F_i(z) - F(z)}^2\right]\leq \sigma^2$, for all $z\in \bb{Z}$,
	where the above expectations are with respect to $i$, with slight abuse of notation.
\end{assumption}
Then, for a carefully chosen tolerance $\sigma_k$, one is able to find a stochastic estimator $\tilde{F}(z^k)$ satisfying the condition in the first line of \eqref{eq-isHalpern} by choosing the sample sizes appropriately. 

\begin{lemma}
	\label{lemma-PAGE}
	Let $S^{(1)}_k$ and $ S^{(2)}_k $ be two random minibatch i.i.d. samples with $\abs{S^{(1)}_k} = N^{(1)}_k$ and $\abs{S^{(2)}_k} = N^{(2)}_k$, and $p_k\in (0,1]$ be a given probability at the $k$-th iteration of \eqref{eq-isHalpern} such that $p_0 = 1$. For $ k\geq 0 $, define
	\[
		\widetilde{F}(z^{k}):= \left\{
			\begin{array}{ll}
				\displaystyle \frac{1}{N^{(1)}_k}\sum_{i \in S^{(1)}_k}F_i(z^k), & \textrm{with probability $ p_k $}, \\
				\displaystyle \widetilde{F}(z^{k-1}) + \frac{1}{N^{(2)}_k}\sum_{i\in S^{(2)}_k}\left(F_i(z^{k}) - F_i(z^{k-1}) \right), & \textrm{with probability $ 1- p_k $}.
			\end{array}
		\right.
	\]
	Suppose that Assumption \ref{assumption-bounded-var} holds. Then, for all $k\geq 1$, it holds that 
	\begin{equation}
	\label{eq-page-1}	
	   \begin{aligned}
	       &\; \bb{E}\left[\norm{\widetilde{F}(z^{k}) - F(z^{k})}^2\right] \\
        \leq &\; \frac{p_k\sigma^2}{N^{(1)}_k} + (1-p_k)\left(\bb{E}\left[\norm{\widetilde{F}(z^{k-1}) - F(z^{k-1})}^2 \right] + \frac{L_0^2}{N^{(2)}_k}\bb{E}\left[ \norm{z^{k} - z^{k-1}}^2\right] \right).
	   \end{aligned}
	\end{equation}
	Moreover, let $\epsilon>0$ and $a>0$ be given parameters, and choose $\sigma_k:= \frac{\epsilon}{(k+1)^a}$, $p_k = 1- \frac{\left(\frac{k}{k+1}\right)^{2a}}{2 - \left(\frac{k}{k+1}\right)^{2a+1}}$, $N^{(1)}_k = \left\lceil\frac{2\sigma^2}{\epsilon^2(k+1)^{-2a}} \right\rceil$, and $N^{(2)}_k = \left\lceil\frac{2L_0^2\norm{z^k - z^{k-1}}^2}{\epsilon^2(k+1)^{-(2a+1)}}\right\rceil$, for $k\geq 0$. (Note that $N^{(2)}_0$ is not needed.) Then, it holds that
	\begin{equation}\label{eq-page-2}
	\bb{E}\left[\norm{\widetilde{F}(z^{k}) - F(z^{k})}^2\right] \leq \sigma_k^2,\quad k\geq 0.	
	\end{equation}
\end{lemma}

The proof of Lemma \ref{lemma-PAGE} can be found in Appendix \ref{proof-lemma-PAGE}. For the rest of this paper, we always assume that $\widetilde{F}$ is constructed according to Lemma \ref{lemma-PAGE} with the same choices of parameters. By choosing $\sigma_k, \gamma_k$ similarly as in Corollary \ref{corollary-epsilon} and Corollary \ref{corollary-large}, we can get similar convergence rates that are presented in the following corollaries, whose proofs are provided in Appendix \ref{proof-cor-E} and Appendix \ref{proof-cor-large-E}, respectively. We can also see that the expected convergence rate given by Corollary \ref{corollary-epsilon-E} is the same as that of \cite{cai2022stochastic}.

\begin{corollary}
	\label{corollary-epsilon-E}
	Suppose that $G$ is $\frac{1}{L}$-co-coercive for a constant $L>0$. Let $ \epsilon \in (0,1] $ be given, $\sigma_k = \gamma_k = \epsilon/\sqrt{k+1}$ for $k\geq 0$ and Assumption \ref{assumption-bounded-var} holds. Suppose that $\{z^k\}$ is the sequence generated by \eqref{eq-isHalpern}. Then, it holds that
	\begin{align*}
	    \bb{E}\left[\norm{G(z^k)}^2\right] \leq &\;  O\left(\frac{1}{(k+1)(k+2)}\right) + O(\epsilon^2),\quad k\geq 0,\\ 
        \bb{E}\left[\norm{z^{k+1} - z^k}^2\right] \leq &\;  O\left(\frac{1}{(k+1)(k+2)}\right) + O(\epsilon^2), \quad k\geq 0.
	\end{align*}
	Hence, after $K:=O(\epsilon^{-1})$ iterations, \eqref{eq-isHalpern} computes an approximate solution $z^K$ such that 
	\[
		\bb{E}\left[\norm{G(z^K)}^2\right] \leq O(\epsilon^2),\quad \bb{E}\left[\norm{z^{K+1} - z^K}^2\right] \leq O(\epsilon^2).
	\]
	Moreover, the sample complexity to get an expected $O(\epsilon)$-optimal solution is $ O\left(\epsilon^{-3}\right) $.
\end{corollary}

\begin{corollary}
	\label{corollary-large-E}
	Suppose that $G$ is $\frac{1}{L}$-co-coercive for a constant $L>0$, the sequence $ \sigma_k = \gamma_k = (k+1)^{-a} $ where $a > \frac{3}{2}$, and Assumption \ref{assumption-bounded-var} holds. Let $\{z^k\}$ be the sequence generated by \eqref{eq-isHalpern}. Then, it holds that
	\[
		\begin{aligned}
		    \bb{E}\left[\norm{G(z^k)}^2\right] \leq &\; O\left(\frac{1}{(k+1)(k+2)}\right),\\
      \bb{E}\left[\norm{z^{k+1} - z^k}^2\right] \leq &\; O\left(\frac{1}{(k+1)(k+2)}\right), 
		\end{aligned}
	\]
	for $k\geq 0$. Hence, after $K:=O(\delta^{-1})$ iterations with $\delta \in (0,1)$, \eqref{eq-isHalpern} computes an approximate solution $z^K$ such that 
	\[
		\bb{E}\left[\norm{G(z^K)}^2\right] \leq O(\delta^2),\quad \bb{E}\left[\norm{z^{K+1} - z^K}^2\right] \leq O(\delta^2).
	\]
	Moreover, the sample complexity to get an expected $O(\delta)$-optimal solution is at most $ O\left(\delta^{-2a}\right) $.
\end{corollary}

We remark again that the variance tolerance $\sigma_k$ in Corollary \ref{corollary-epsilon-E} (and \cite{cai2022stochastic}) is roughly chosen as $O(\epsilon/\sqrt{k})$ with a specified tolerance $\epsilon > 0$. In the case when $\epsilon$ is small, one would need  large sample sizes in the early iterations of the Halpern iterative scheme; See also \cite[Corollary 2.2]{cai2022stochastic}. On the contrary, our result enables us to start with larger tolerances, hence it may require smaller sample sizes in the early iterations.

\section{Data-driven Wasserstein distributionally robust optimization} 
\label{sec:WDRO}

In this section, we are interested in the following Wasserstein distributionally robust optimization (WDRO) \cite{zhaoDatadrivenRiskaverseStochastic2018,mohajerin2018data}:
\begin{equation}
\label{eq-wdro}
		\min_{x \in \mathcal{X}}\; \sup_{\mu\in \bb{B}_\theta(\hat{\mu}_N)}\; \bb{E}_{\xi\sim \mu} [\ell(x,\xi)],
\end{equation}
where the ambiguity set is defined as the Wasserstein ball centered at $\hat{\mu}_N$ with radius $\theta > 0$. In particular, $\bb{B}_\theta(\hat{\mu}_N):=\{ \mu\in \mathcal{P}(\Xi): \int_{\Xi\times \Xi}c(\xi_1,\xi_2)\pi(\dd \xi_1, \dd \xi_2) \leq \theta,\; \pi \in \Pi(\mu,\hat{\mu}_N)\}$, where $\mathcal{P}(\Xi)$ denotes the space of all distributions supported on $\Xi$, $\Pi(\mu,\nu )$ denotes the joint distribution over $\Xi\times \Xi$ with marginals $\mu$ and $\nu$, respectively, and $c:\Xi\times\Xi\to\overline{\bb{R}}$ is the transportation cost function. 

It is worth noting that robustness promoted by WDRO \eqref{eq-wdro} comes with a price in the sense that evaluating the worst-case expected loss 
\begin{equation}
\label{eq-WDRO-inner-loss}
    \sup_{\mu\in \bb{B}_\theta(\hat{\mu}_N)}\; \bb{E}_{\xi\sim \mu} [\ell(x,\xi)]
\end{equation}
is difficult in general since it involves functional variables.
	In fact, even checking whether a given distribution $ \mu \in \mathcal{P}(\Xi) $ belongs to $ \bb{B}_\theta(\hat{\mu}_N) $ is already challenging \cite{tacskesen2022semi}. Therefore, unless in some special situations  (see e.g., \cite{shafieezadeh2018wasserstein}), solving \eqref{eq-wdro} directly is impractical. Fortunately, under mild conditions, \eqref{eq-WDRO-inner-loss} admits a dual problem and the strong duality holds \cite{blanchet2019quantifying}. Utilizing the resulting dual problem and the strong duality, it is possible to reformulate the original WDRO problem \eqref{eq-wdro} as a 
{finite-dimensional} convex program, under additional conditions such as the convexity and concavity of $\ell$ in its first and second arguments, respectively, convexity of the feasible set $\mathcal{X}$ and the support set $\Xi$, and a certain growth conditions of $\ell$ with respect to the second argument \cite{shafieezadeh2015distributionally,lee2015distributionally,mohajerin2018data,gao2022distributionally}. To the best of {our} knowledge, most existing works relied on special structure of the loss function $\ell$ such that the reformulated finite convex program can be solved via off-the-shelf solvers such as Gurobi, Mosek, CPLEX and IPOPT \cite{wachter2006implementation}. Although these solvers have shown their high efficiency in computing accurate solutions robustly for small-to-medium-scale problems, they inevitably consume too much computational resource when solving large-scale problems. Thus, when the size of the training data set, $N$, is large, these solvers become less attractive or even inapplicable to solve the reformulated convex program.
Our method can operate in the inexact and stochastic settings, and is not limited by the loss function types $\ell$ previously considered.

To alleviate the curse of dimensionality, recent years have witnessed growing interests in applying (deterministic and/or stochastic) first-order methods for solving the WDRO problem \eqref{eq-wdro} and its special cases. Next, we mention some of the recent works along this direction. \cite{shafieezadeh2018wasserstein} applied a Frank-Wolfe algorithm for solving the nonlinear semidefinite programming reformulation of the Wasserstein distributionally robust Kalman filtering problem where the ambiguity set contains only normal distributions. One of the appealing features of the Frank-Wolfe algorithmic framework is that the corresponding search direction can be found in a quasi-closed form. \cite{li2019first} proposed a linear proximal alternating direction method of multipliers (ADMM) for Wasserstein distributionally robust logistic regression and established a sublinear convergence rate in terms of objective function values. Later, an epigraphical projected-based incremental method was applied for solving the Wasserstein distributionally robust support vector machine \cite{li2020fast}. The convergence properties of the previous method were also established by using a H\"{o}lderian growth condition with an explicit growth exponent. In \cite{luo2019decomposition}, a class of Wasserstein distributionally robust regression problems were reformulated as decomposable semi-infinite programs and a cutting-surface method was applied and analyzed. Stochastic gradient descent (SGD) method was also applied for solving a class of Wasserstein distributionally robust supervised learning (WDRSL) problems with a locally strongly convex loss function \cite{blanchet2022optimal}. The SGD method was also applied for training the optimal transport-based distributionally robust semi-supervised learning tasks in \cite{blanchet2020semi}. The corresponding sample complexity and iteration complexity were also provided. More recently, \cite{yu2022fast} considered formulating WDSRLs as structural min-max optimization problems which were solved by stochastic extra-gradient algorithms. In \cite{yu2022fast}, the ideas of variance reduction and random reshuffling were investigated for solving the corresponding min-max optimization problem.

Following the research theme in solving WDRO problems, in this section, we consider two important classes of WDRO problems that admit tractable convex or convex-concave minimax programming reformulations and hence can then be reformulated as monotone inclusion problems. The first class of WDRO problems assumes that the loss function $\ell$ is chosen as the generalized linear model, and the second problem class considers the case when $\ell$ is convex-concave. 

\subsection{Wasserstein distributionally robust supervised learning} 
\label{sec:WDRSL}
The training data set in a classical supervised learning task typically contains pairs of data defining features and their supervision. We consider the task of binary classification, where the support set $\Xi$ (also called the feature-label space) has the form $\bb{R}^{d-1}\times \{-1,1\}$ and the training data set can be expressed as $\left\{\hat\xi_i:=(\hat\phi_i,\hat\psi_i)\right\}_{i = 1}^N$ with $\hat\phi_i\in \bb{R}^{d-1}$ and $\hat\psi_i\in \{-1,1\}$.  We consider the following Wasserstein distributionally robust supervised learning problem:
\begin{equation}
\label{eq-WDRSL}	
	\min_{w\in \Gamma}\;\sup_{\mu\in \bb{B}_\theta(\hat{\mu}_N)}\;\bb{E}_{\xi:=(\phi,\psi)\sim\mu}\left[\Psi_0(w) + \Psi(\inprod{\phi}{w}) - \psi\inprod{\phi}{w}\right]
\end{equation}
where $ \Gamma\subseteq \bb{R}^{d-1} $ is a closed convex set, $\Psi_0,\Psi:\bb{R}\to \bb{R}$ are $\overline{L}_0$-smooth, $\widetilde{L}_0$-Lipschitz continuous and convex functions with $\bar{L}_0,\widetilde{L}_0 > 0$. Note that in this paper, we consider the more general setting with $w\in \Gamma$, while in \cite{shafieezadeh2015distributionally,yu2022fast}, $\Gamma = \bb{R}^{d-1}$. It is obvious that one can incorporate prior information into a decision by selecting a specific option for the set $\Gamma$. The function $\Psi_0$ in \eqref{eq-WDRSL}  serves as a regularizer and is often used in machine learning applications. For instance, a popular choice for the regularizer is $\Psi_0(w) = \frac{\rho}{2}\norm{w}^2$ where $\rho \geq 0$ is the regularization parameter. Moreover, we can see that model \eqref{eq-WDRSL} is quite general since it covers many important real-world applications, including the Wasserstein distributionally robust Logistic regression \cite{shafieezadeh2015distributionally,li2019first}.

We define the transportation cost function $c:\Xi\times \Xi \to \bb{R}$ as 
\[
c(\xi_1,\xi_2):= \norm{\phi_1 - \phi_2} + \kappa \abs{\psi_1 - \psi_2},\quad \xi_1:=(\phi_1,\psi_1),\; \xi_2:= (\phi_2,\psi_2),
\]
where the parameter $ \kappa>0 $ denotes the relative importance of feature mismatch and label uncertainty. Then, we can show in the following proposition that problem \eqref{eq-WDRSL} is equivalent to a convex-concave min-max optimization problem.  The reformulation is derived mostly by combining previously published results; See e.g., \cite{mohajerin2018data,shafieezadeh2015distributionally,yu2022fast}. And the key step for deriving the convex-concave reformulation \eqref{eq-WDRSL-minmax} is to consider the dual reformulation of the inner maximization problem for a fixed decision variable $\beta$. Since the proof of the proposition can be done word-by-word as in \cite{yu2022fast}, we omit it here for simplicity.
\begin{proposition}
	\label{prop-WDRSL-minmax}
	The Wasserstein distributionally robust learning with generalized linear model \eqref{eq-WDRSL} is equivalent to
	\begin{equation}
	\label{eq-WDRSL-minmax}
	\min_{x:=(w,\lambda) \in \bb{R}^d}\; \max_{y \in \bb{R}^N}\; f(x,y):=\frac{1}{N}\sum_{i = 1}^Nf_{i}(x,y) \quad 
		 \mathrm{s.t.}\quad  x\in \mathcal{X},\; y\in \mathcal{Y},
	\end{equation}
	where $f:\bb{R}^d\times \bb{R}^N\to \bb{R}$, $\mathcal{X}\subseteq \bb{R}^d$ and $\mathcal{Y}\subseteq \bb{R}^N$ are defined as:
\[
\left\{
	 \begin{aligned}
	 	f_{i}(x,y):=  &\; \Psi_0(w) + \lambda(\theta - \kappa) + \Psi\left(\inprod{\hat\phi_i}{w}\right) + y_i \left(\hat\psi_i\inprod{\hat\phi_i}{w} - \lambda \kappa\right),\; i \in [N], \\
	 	\mathcal{X}:= &\; \left\{x:=(w,\lambda)\in\bb{R}^d: \norm{w}\leq \frac{\lambda}{\widetilde{L}_0+1},\; \omega\in \Gamma\right\}, \\
	 	\mathcal{Y}:= &\;\left\{y\in \bb{R}^N: \norm{y}_\infty \leq 1\right\}.
	 \end{aligned} 
	 \right.
\]
\end{proposition}

By the definition of the function $f_{i}$ for $i = 1, \dots, N$, we see that 
\[
	\nabla f_{i}(x,y) = 
	\begin{pmatrix}
		\Psi_0'(w)+\Psi'\left(\inprod{\hat\phi_i}{w}\right)\hat\phi_i + y_i \hat\psi_i \hat\phi_i \\
		\displaystyle \theta - \kappa -  y_i\kappa \\
		\left(\hat\psi_i\inprod{\hat\phi_i}{w} - \lambda \kappa\right) e_i
	\end{pmatrix},\;  x:=(w,\lambda) \in \bb{R}^d,\; y \in \bb{R}^N,
\]
where $e_i$ is the $i$-th standard basis vector for $\bb{R}^N$. Using the expression of $\nabla f_{i}$, one can check that $f_i$ has a Lipschitz continuous gradient (see Lemma \ref{lemma-fi-Lip}) with a common modulus $L_0>0$, for $\forall i$.

Recall that $f_{i}$ is convex-concave, and Lemma \ref{lemma:cvxcoco} implies that the mapping 
\[
	F_{i}(z):= F_{i}(x,y):= \begin{pmatrix}
		\nabla_x f_{i}(x,y) \\
		-\nabla_y f_{i}(x,y)
	\end{pmatrix} = \begin{pmatrix}
		\Psi_0'(w) + \Psi'\left(\inprod{\hat\phi_i}{w}\right)\hat\phi_i + y_i \hat\psi_i \hat\phi_i \\
		\displaystyle \theta - \kappa - \kappa y_i \\
		-\left(\hat\psi_i\inprod{\hat\phi_i}{w} - \lambda \kappa\right) e_i
	\end{pmatrix},
\]
is $\frac{1}{L_0}$-co-coercive. Hence, the problem \eqref{eq-WDRSL-minmax} is a convex-concave min-max optimization problem with a finite-sum structure that can be solved via applying the presented Halpern iterations for the mapping $G$ defined similarly as in \eqref{eq-finitesum-G}:
\begin{equation}
	\label{eq-finitesum-G-WDRSL}
	0\in G(z) = G(x,y):= \frac{1}{\alpha} \begin{pmatrix}
		\displaystyle x - \proj{\mathcal{X}}{x - \frac{\alpha}{N}\sum_{i = 1}^N\nabla_x f_{i}(x,y)} \\
		\displaystyle y - \proj{\mathcal{Y}}{y +  \frac{\alpha}{N}\sum_{i = 1}^N \nabla_y f_i(x,y)} \\
	\end{pmatrix} ,
\end{equation}
where $z := (x,y)\in \bb{Z}:=\mathbb{R}^d\times \mathbb{R}^N$ and $ \alpha\in \left(0, \frac{4}{L_0}\right) $ is a given parameter. Note that we are interested in the case when the number of data points, i.e., $N$, is extremely large, so that evaluating the function $f$ and its gradient can be expensive. {Thus, it is preferable to approximate the finite sum by a certain stochastic estimator.} By applying Corollary \ref{corollary-epsilon-E} and Corollary \ref{corollary-large-E}, we immediately get the following result.

\begin{corollary} 
    \label{corollary-WDRSL}
    Let $\{z^k\}$ be the sequence generated by \eqref{eq-isHalpern} with the mapping $G$ given by \eqref{eq-finitesum-G-WDRSL} and $ \epsilon \in (0,1] $ be given. Suppose that Assumption \ref{assumption-bounded-var} holds and one of the following two conditions holds: (1) $\sigma_k = \gamma_k = \epsilon/\sqrt{k+1}$ for $k\geq 0$; (2) $ \sigma_k = \gamma_k = (k+1)^{-a} $ where $a > \frac{3}{2}$. Then, after $K:=O(\epsilon^{-1})$ iterations, \eqref{eq-isHalpern} computes an approximate solution $z^K$ such that 
	\[
		\bb{E}\left[\norm{G(z^K)}^2\right] \leq O(\epsilon^2),\quad \bb{E}\left[\norm{z^{K+1} - z^K}^2\right] \leq O(\epsilon^2).
	\]
	Moreover, the sample complexity to get an expected $O(\epsilon)$-optimal solution is at most $O\left(\epsilon^{-3}\right)$ and $ O\left(\epsilon^{-2a}\right) $, respectively for the above two conditions.
\end{corollary}

By the definitions of $\mathcal{X}$ and $\mathcal{Y}$, we can see that the projection onto the set $\mathcal{Y}$ can be performed analytically, i.e., 
\[
	\proj{\mathcal{Y}}{y} = 
 \big({\rm sgn}(y_i) \min\{1,|y_i|\}\big)_{i=1}^N,  \quad \forall \; y\in \bb{R}^{N}.
\]
Moreover, when $\Gamma = \bb{R}^{d-1}$, the set $\mathcal{X}$ reduces to the so-called ice cream cone and also admits an analytical expression \cite{bauschke1996projection}, i.e., 
\[
	\proj{\mathcal{X}}{x} = 
	\left\{
		\begin{array}{ll}
			(w, \lambda), & \textrm{if $ \norm{w}\leq \frac{\lambda}{\widetilde{L}_0 + 1} $}, \\
			(0,0), & \textrm{if $ \frac{1}{\widetilde{L}_0 + 1}\norm{w}\leq -\lambda  $}, \\
			\frac{\frac{1}{\widetilde{L}_0 + 1}\norm{w} + \lambda}{\frac{1}{(\widetilde{L}_0 + 1)^2} +1}\left(\frac{1}{(\widetilde{L}_0 + 1)\norm{w}}w, 1 \right), & \textrm{otherwise},
		\end{array}
		\right.,
\]
for all $x:=(w, \lambda)\in \bb{R}^d$. However, when $\Gamma\neq \bb{R}^{d-1}$, $\proj{\mathcal{X}}{\cdot}$ may not be evaluated exactly in general. In this case, one needs to rely on iterative schemes and is only able to obtain an approximate projection. Nevertheless, the proposed inexact Halpern iterations ensures the desirable convergence properties under mild conditions.

\subsection{{WDRO} with convex-concave loss functions}\label{sec:wdro-cc}
In this subsection, we consider the $p$-Wasserstein DROs $ (p\geq 1) $ with convex-concave loss functions. In particular, the loss function $\ell:\mathcal{X}\times \Xi\to \overline{\bb{R}}$ is assumed to be $L_0$-smooth and convex-concave where $L_0>0$, and $\mathcal{X}\subseteq \bb{R}^n$ is the feasible region and $\Xi\subseteq \bb{R}^d$ denotes the support set, and the transportation cost is chosen as $ c(\xi',\xi) := d^p(\xi'
, \xi) $, for all $ \xi',\xi\in \Xi $, where $d$ is a metric on $\Xi$. For the rest of this section, we assume that $\mathcal{X}$ is a convex compact set and $(\Xi,d)$ is a Polish space. We note that we do not assume the loss function $\ell$ is necessarily of one of the special types considered in previous works such as quadratic or logistic. Therefore, our method is more general in the sense that we propose an inexact stochastic solver for general nonlinear convex-concave WDRO problems.

From \cite{gao2022distributionally} (see also Lemma \ref{lemma-polish-dual}), we see that the WDRO problem \eqref{eq-wdro} is equivalent to the following min-max optimization problem:
\begin{equation}
\label{eq-minmax-wdro}
\min_{x\in\bb{R}^n}\; \max_{\xi_i \in \bb{R}^d}\; \frac{1}{N}\sum_{i = 1}^N\ell(x,\xi_i) \quad \mathrm{s.t.}\quad 	x\in \mathcal{X}, \; \sum_{i=1}^Nd^p(\xi_i, \hat{\xi}_i) \leq N\theta^p,\; \xi_i\in \Xi,\; i \in[N].
\end{equation}
For notational simplicity, let us denote $\hat{y}:=[\hat{\xi}_1;\dots;\hat{\xi}_N]\in \bb{R}^{Nd}$ as the training data set, and $y:=[\xi_1;\dots, \xi_N]\in \bb{R}^{Nd}$. Then problem \eqref{eq-minmax-wdro} can be written compactly as 
\[
\min_{x\in\mathcal{X}}\;\max_{y\in\mathcal{Y}}\; \frac{1}{N}\sum_{i=1}^N f_i(x,y),
\]
where $f_i(x,y):= \ell(x, \xi_i),\; i \in[N]$, and 
\[
\mathcal{Y}:=  \left\{y=[\xi_1;\dots;\xi_N]: \sum_{i=1}^Nd^p(\xi_i, \hat{\xi}_i) \leq N\theta^p,\; \xi_i\in \Xi,\; i \in [N]\right\}.
\]
Since $\ell$ is $L_0$-smooth and convex-concave, similar to Section \ref{sec:WDRSL}, the proposed inexact Halpern schemes can be applied to solve the above convex-concave min-max optimization problem by reformulating it as a finite-sum inclusion problem as follows:
\begin{equation}
    \label{eq-G-WDRO-cvx-loss}
    0 \in G(z) = G(x, y) := \frac{1}{\alpha} \begin{pmatrix}
		\displaystyle x - \proj{\mathcal{X}}{x - \frac{\alpha}{N}\sum_{i = 1}^N\nabla_x f_{i}(x,y)} \\
		\displaystyle y - \proj{\mathcal{Y}}{y +  \frac{\alpha}{N}\sum_{i = 1}^N \nabla_y f_i(x,y)} \\
	\end{pmatrix}.
\end{equation}
Moreover, in the particular cases when $N$ is large and the sets $\mathcal{X}$ and $\mathcal{Y}$ do not guarantee explicit projections, \eqref{eq-isHalpern} {with the proposed variance reduced stochastic estimator} becomes useful. Similar to Corollary \ref{corollary-WDRSL}, we have the following convergence results when applying \eqref{eq-isHalpern} for solving \eqref{eq-G-WDRO-cvx-loss}.

\begin{corollary} 
    \label{corollary-WDRO-cvx-loss}
    Let $\{z^k\}$ be the sequence generated by \eqref{eq-isHalpern} with the mapping $G$ given by \eqref{eq-G-WDRO-cvx-loss} and $ \epsilon\in (0,1] $ be given. Suppose that Assumption \ref{assumption-bounded-var} holds and one of the following two conditions holds: (1) $\sigma_k = \gamma_k = \epsilon/\sqrt{k+1}$ for $k\geq 0$; (2) $ \sigma_k = \gamma_k = (k+1)^{-a} $ where $a > \frac{3}{2}$. Then, after $K:=O(\epsilon^{-1})$ iterations, \eqref{eq-isHalpern} computes an approximate solution $z^K$ such that 
	\[
		\bb{E}\left[\norm{G(z^K)}^2\right] \leq O(\epsilon^2),\quad \bb{E}\left[\norm{z^{K+1} - z^K}^2\right] \leq O(\epsilon^2).
	\]
	Moreover, the expected sample complexity to get an expected $O(\epsilon)$-optimal solution is at most $O\left(\epsilon^{-3}\right)$ and $ O\left(\epsilon^{-2a}\right) $, respectively for the above two conditions.
\end{corollary}

\section{Numerical experiments} \label{sec:numerical}

In this section, we conduct preliminary numerical experiments to validate the theoretical development of the proposed inexact Halpern iterations. To this end, we solve the 2-Wasserstein DRO problems \eqref{eq-wdro} with the convex-concave loss that can be either quadratic-linear or nonlinear. The code for these experiments is available at:
\begin{center}
    \url{https://github.com/liangling98/isHalpern/tree/main}
\end{center}  

\subsection{2-Wasserstein DRO problems with quadratic-linear loss}
We first solve the DRO problem with a quadratic-linear loss function that is given as
\[
    \ell(x,\xi):= \frac{1}{2}\norm{Ax - \xi}^2 - \frac{1}{2}\|\xi\|^2,\quad \forall\; x \in \mathcal{X},\; \xi \in \Xi,
\]
where $A\in \mathbb{R}^{d\times n}$ is given with $d \leq n$, $\mathcal{X}:=\{ x \in \mathbb{R}^n \;:\; e^T x = 1,\; \texttt{lb} \leq x \leq \texttt{ub}\}$ with given lower and upper bounds satisfying $-\infty \leq \texttt{lb} \leq \texttt{ub} \leq \infty$, and $\Xi:= \mathbb{R}^d$ with a given parameter $\rho > 0$. Note that 
\[
    \nabla_x \ell(x,\xi) = A^T(Ax - \xi),\quad \nabla_\xi \ell(x,\xi) =  - Ax.
\]
One can verify that $\ell$ is $L_0$-smooth, where $L_0 = \norm{[A^TA, -A^T; -A, 0]}_2$.

According to \eqref{eq-minmax-wdro}, the resulted DRO can be reformulated as the min-max optimization problem:
\[
    \min_{x\in \mathcal{X}}\; \max_{\xi\in \mathcal{Y}} \quad \frac{1}{2N}\sum_{i=1}^N \left(\norm{Ax - \xi_i}^2 - \|\xi_i\|^2\right), 
\]
where 
\[
\mathcal{Y}:=\left\{y:=(\xi_1,\dots,\xi_N)^T\in \mathbb{R}^{N d}\;:\; \norm{y - \hat{y}} \leq \sqrt{N}\theta\right\}, \quad \hat{y}:=(\hat{\xi}_1,\dots,\hat{\xi}_N)^T.
\]
This problem can be further modeled as an inclusion problem $0\in G(x,y)$ as in \eqref{eq-G-WDRO-cvx-loss}, where 
\[
    G(x,y):= \frac{1}{\alpha} \begin{pmatrix}
		\displaystyle x - \proj{\mathcal{X}}{x - \alpha A^TAx + \frac{\alpha}{N}\mathcal{A}^Ty} \\
		\displaystyle y - \proj{\mathcal{Y}}{y -\frac{\alpha}{N} \mathcal{A}x} \\
	\end{pmatrix}, \quad \forall\; (x,y)\in \mathbb{R}^n\times \mathbb{R}^{Nd},
\]
and 
\[
\alpha \in \left(0, \frac{4}{L_0}\right), \quad  \mathcal{A}:= \begin{pmatrix}
	    A \\ \vdots \\ A
	\end{pmatrix} \in \mathbb{R}^{Nd\times n}.
\]
From Lemma \ref{lemma:G-coco}, we see that $G$ is $\frac{\alpha(4-\alpha L_0)}{4}$-co-coercive. In this section, we always set $\alpha = 2/L_0$ for simplicity. 

Next, we provide a detailed description of the experimental settings used in our numerical tests. 

\textbf{Data set.} Given $d$ and $n$, we generate a random matrix $A\in \mathbb{R}^{d\times n}$ with entries sampled from the uniform distribution (i.e., $A = \texttt{rand}(d,n))$. The columns of $A$ are then normalized to have unit Euclidean norms. Next, we construct a random vector $\tilde{x}\in \mathbb{R}^n$, also drawn from the uniform distribution, and normalize it so that its elements sum to one. Using $\tilde{x}$, we define the lower and upper bounds as:
\[
    \texttt{lb} = \tilde{x} - \frac{1}{4}\texttt{rand}(n),\quad  \texttt{ub} = \tilde{x} + \frac{1}{4}\texttt{rand}(n).
\]
This guarantees that the set $\mathcal{X}$ is nonempty. For the training data set, we generate $\hat{y} = \texttt{rand}(Nd)$ with $N = 200d$, and set the radius of the Wasserstein ball to $\theta = 10^{-2}$. Finally, we terminate the algorithm when the total number of gradient queries to $\ell$ reaches $50N$.

\textbf{Baseline solver.} We evaluate the performance of the following four algorithms: the exact Halpern iteration (HI), the deterministic inexact Halpern iteration (iHI) with approximate evaluations of $G$, the inexact stochastic Halpern iteration (isHI) with samples sizes $N^{(1)}_k = N^{0.7}$ and $N^{(1)}_k = N^{0.3}$ for all $k\geq 0$, and the classical projected gradient descent-ascent method (PGDA). Based on our numerical experience, we observed that the extragradient method \cite{cheney1959proximity} and Popov's method \cite{popov1980modification} exhibit performance comparable to PGDA. Therefore, for simplicity, we present only the results for PGDA. 

\textbf{Projections onto $\mathcal{X}$ and $\mathcal{Y}$.} While the projection onto $\mathcal{Y}$ can be evaluated exactly, the projection onto $\mathcal{X}$ requires an iterative solver for general $\texttt{lb}$ and $\texttt{ub}$ \cite{liang2024vertex}. In our experiments, we use the alternating projection algorithm \cite{cheney1959proximity} to approximate the projection of a given point onto $\mathcal{X}$. The algorithm is terminated when the following conditions are satisfied for a specified tolerance $\texttt{tol} \geq 0$:
\[
    |e^Tx - 1|< \texttt{tol}, \quad x > \texttt{lb}-\texttt{tol}, \quad x < \texttt{ub} + \texttt{tol}.
\]
We set $\texttt{tol} = 10^{-12}$ to compute a nearly exact projection for the HI and PGDA methods. For the proposed inexact Halpern iterations, we use $\texttt{tol} = 5\times 10^{-2} / \sqrt{k+1}$ where $k$ denotes the iteration counter, to establish an inexact projection setting.

The computational results for various values of $d$ and $n$ are presented in Figures \ref{fig:1}–\ref{fig:8}. The legends of these figures also report the average number of iterations taken by the alternating projection method for four solvers.

First, we observe that by leveraging the relaxed inexactness conditions established in this paper, iHI achieves performance comparable to that of HI while requiring significantly less effort to approximate the projection onto $\mathcal{X}$. This finding validates both the correctness and the effectiveness of our theoretical results. Second, when comparing the performance of PGDA and HI, we note that PGDA typically exhibits a faster convergence rate during the initial iterations but slows down in later stages. This behavior suggests a practical strategy: running PGDA for a few iterations to generate a good-quality initial point for Halpern iterations can enhance overall efficiency. Lastly, for large $N$, isHI demonstrates clear efficiency gains, aligning with the prevalent use of stochastic algorithms in machine learning tasks.

Overall, our numerical results strongly support the motivation for considering inexact and/or stochastic Halpern iterations with relaxed inexactness conditions and variance reduction techniques. These findings further inspire potential advanced applications in related fields.

\begin{figure}[!]
    \centering

    \begin{subfigure}[b]{0.45\textwidth}
        \centering
        \includegraphics[width=\textwidth]{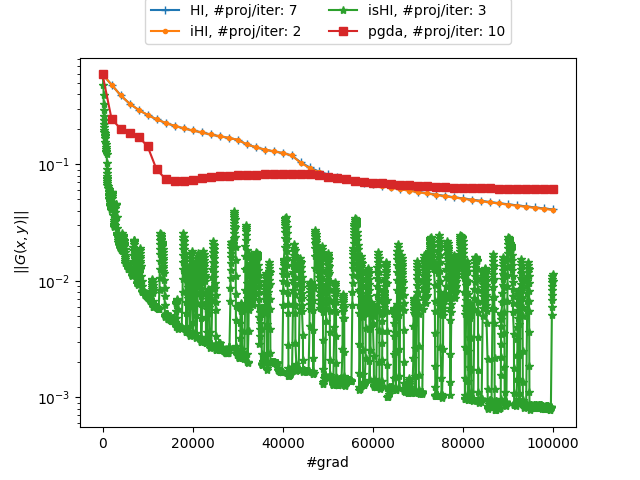}
        \caption{$d = 10$, $n = 10$}
        \label{fig:1}
    \end{subfigure}
    \hfill
    \begin{subfigure}[b]{0.45\textwidth}
        \centering
        \includegraphics[width=\textwidth]{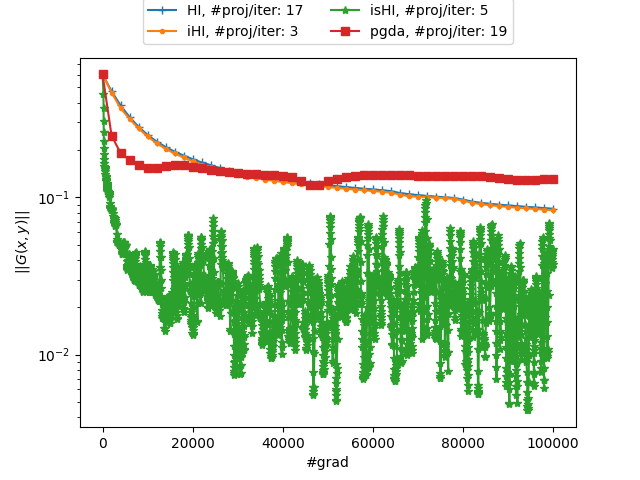}
        \caption{$d = 10$, $n = 20$}
        \label{fig:2}
    \end{subfigure}

    \vspace{0.5em}

    \begin{subfigure}[b]{0.45\textwidth}
        \centering
        \includegraphics[width=\textwidth]{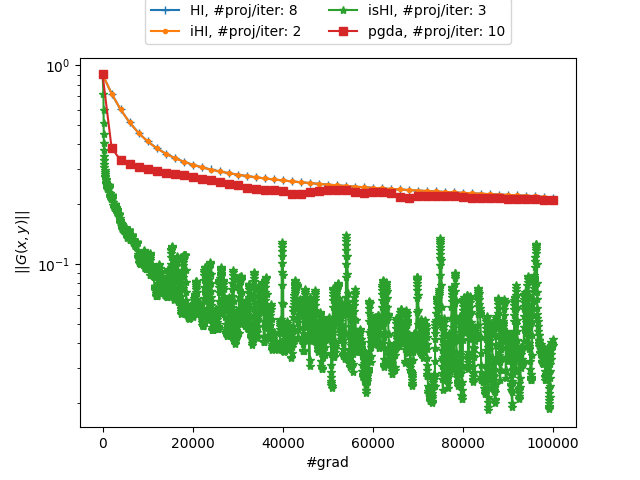}
        \caption{$d = 10$, $n = 50$}
        \label{fig:3}
    \end{subfigure}
    \hfill
    \begin{subfigure}[b]{0.45\textwidth}
        \centering
        \includegraphics[width=\textwidth]{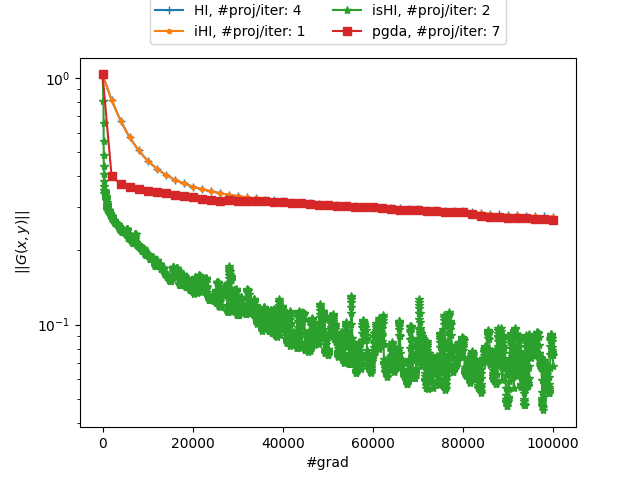}
        \caption{$d = 10$, $n = 100$}
        \label{fig:4}
    \end{subfigure}

    \vspace{0.5em}

    \begin{subfigure}[b]{0.45\textwidth}
        \centering
        \includegraphics[width=\textwidth]{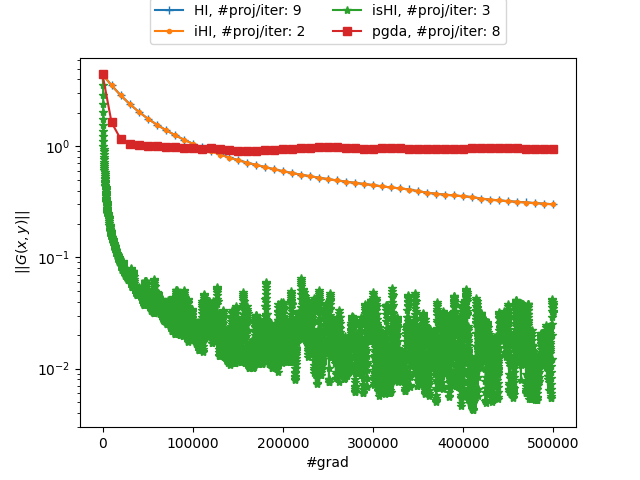}
        \caption{$d = 50$, $n = 50$}
        \label{fig:5}
    \end{subfigure}
    \hfill
    \begin{subfigure}[b]{0.45\textwidth}
        \centering
        \includegraphics[width=\textwidth]{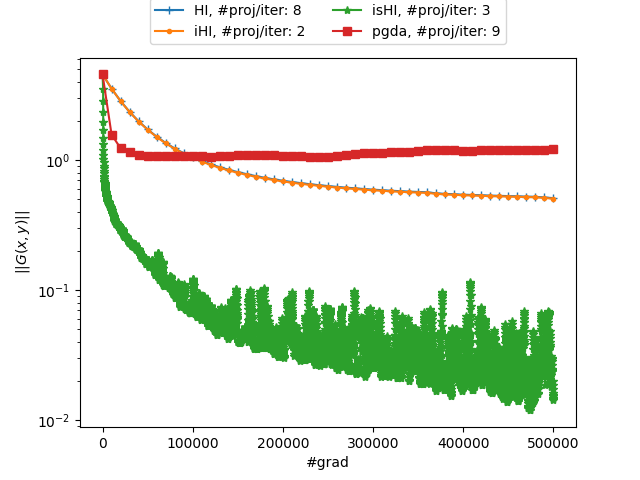}
        \caption{$d = 50$, $n = 100$}
        \label{fig:6}
    \end{subfigure}

    \vspace{0.5em}

    \begin{subfigure}[b]{0.45\textwidth}
        \centering
        \includegraphics[width=\textwidth]{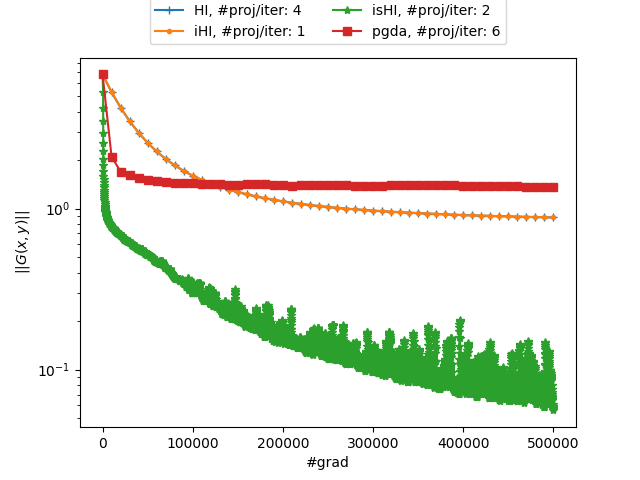}
        \caption{$d = 50$, $n = 250$}
        \label{fig:7}
    \end{subfigure}
    \hfill
    \begin{subfigure}[b]{0.45\textwidth}
        \centering
        \includegraphics[width=\textwidth]{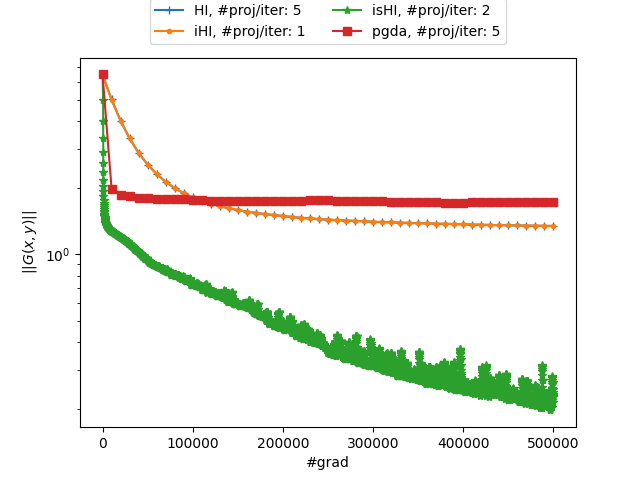}
        \caption{$d = 50$, $n = 500$}
        \label{fig:8}
    \end{subfigure}

    \caption{Results for different values of $d$ and $n$.}
    \label{fig:combined-2x4}
\end{figure}

\subsection{2-Wasserstein DRO problems with nonlinear convex-concave loss}
In this part, we utilize the inexact stochastic Halpern iteration to address a 2-Wasserstein Distributionally Robust Optimization problem featuring a nonlinear convex-concave loss function. The loss function is defined as
\[
\ell(x,\xi) := \frac{1}{2} x^{T}x - x^{T}\xi - \frac{1}{2} \exp\left(-\xi^{T}\xi\right) - \frac{1}{2} \xi^{T}\xi, \quad x \in \mathcal{X},\; \xi \in \Xi,
\]
where $\mathcal{X} := \mathbb{R}^n$ and $\Xi := \mathbb{R}^n$. The DRO problem is formulated as
\begin{equation}\label{problem:NDRO}
\min_{x\in\mathcal{X}} \sup_{\mu\in \bb{B}_\theta(\hat{\mu}_N)}\; \bb{E}_{\xi\sim \mu}\left[\ell(x,\xi)\right],
\end{equation}
with a given radius $\theta > 0$ for the Wasserstein ball $\bb{B}_\theta(\hat{\mu}_N)$ centered at the empirical distribution $\hat{\mu}_N$ derived from $N$ samples $\hat{\xi}_1, \dots, \hat{\xi}_N$.

As detailed in Section \ref{sec:wdro-cc}, the problem \eqref{problem:NDRO} can be reformulated as the min-max optimization problem:
\begin{equation}\label{problem:SAA NDRO}
\min_{x\in\mathcal{X}} \sup_{y\in\mathcal{Y}_{\hat{y}}^\theta}\frac{1}{N}\sum_{i=1}^N \ell(x,\xi_i),
\end{equation}
where $y:=(\xi_1,\dots,\xi_N)^T \in \mathbb{R}^{Nn}$, $\hat{y}:=(\hat{\xi}_1,\dots,\hat{\xi}_N)^T \in \mathbb{R}^{N*n}$, and the feasible set for $y$ is
\[
\mathcal{Y}_{\hat{y}}^\theta := \left\{ y \in \mathbb{R}^{Nn} \;:\; \norm{y - \hat{y}}^2 \leq N\theta^2 \right\}.
\]
Let $L(x,y) = \frac{1}{N}\sum_{i=1}^N \ell(x,\xi_i)$.

A key difficulty in applying the standard dual reformulation approach to WDRO is that it requires computing the convex conjugate of $-\ell(x,\xi)$ with respect to $\xi$ \cite{mohajerin2018data}. Since no closed‐form expression for its conjugate is available, the dual method cannot be implemented in this case. Fortunately, $L(x,y)$ meets all the prerequisites for employing Halpern iteration schemes. It is convex in $x$, as its second partial derivative with respect to $x$ is $\frac{\partial ^2 L}{\partial x^2}(x,y) = \frac{1}{N}\sum \frac{\partial ^2 l}{\partial x^2}(x,\xi_i) = I_n$, where $I_n$ is the $n \times n$ identity matrix. For the $y$ variable, $\frac{\partial ^2 L}{\partial y^2}(x,y) = \text{diag}(\frac{\partial ^2 \ell}{\partial \xi^2}(x,\xi_1), \frac{\partial ^2 \ell}{\partial \xi^2}(x,\xi_2), ..., \frac{\partial ^2 \ell}{\partial \xi^2}(x,\xi_N))$. A simple computation shows that $\frac{\partial ^2 \ell}{\partial \xi^2}(x,\xi_i) = (I_n - \xi_i \xi_i^T)\exp(-\xi_i^T \xi_i) - I_n$. Since $\exp(-\xi_i^T \xi_i) \in (0, 1]$ for all $\xi_i \in \mathbb{R}^n$, $\exp(-\xi_i^T \xi_i) I_n - I_n$ is negative semi-definite. And $ \xi_i \xi_i^T$ is always positive semi-definite, therefore $\frac{\partial ^2 \ell}{\partial \xi^2}(x,\xi_i) \preceq 0, i = 1, ..., N$, consequently, $L(x,y)$ is concave in $y$. Because the spectral norm of the Hessian matrix of $\ell(x,\xi)$ is less than 2 for all $(x,\xi)$, its gradient $\nabla \ell$ is $1/2$-co-coercive. Hence $L$ is $1/2$-co-coercive as well. $L(x,y)$ also satisfies Assumption \ref{assumption-bounded-var}. Consequently, all conditions for the application of Halpern iterations are met. 

Next, we detail the experimental setup for our numerical experiments. 

\textbf{Data Generation.}
We selected the parameters $n=3$, $N=100$, and $\theta = 0.1$. The $N$ empirical samples $\hat{\xi}_i \in \mathbb{R}^n$ comprising $\hat{y}$ were drawn i.i.d. from a standard normal distribution $\mathcal{N}(0,I_3)$. The initial primal variable $x_0$ was generated by sampling uniformly from $[0,1]^n$. The initial distribution was represented by the original sampled points.

\textbf{Algorithm Settings.}
We implemented isHI with a step size parameter $\alpha=1$. The stochastic gradients were computed using the PAGE variance-reduced estimator, configured with parameters $\epsilon = 0.01$, $a = 2$, and $\sigma = 1$. The operator $G(x,y)$ for the inclusion problem was formulated based on \eqref{eq-G-WDRO-cvx-loss}. The iterative process was terminated when the norm of the gradient mapping, $\norm{G(x^k,y^k)}$, fell below a threshold of $5 \times 10^{-3}$.

Figure \ref{fig:NDRO_loss} presents the evolution of the test loss $L(x^k,y^*)$, where $y^*=\{\xi^*_1, ...,\xi^*_N\}$ denotes the worst-case perturbed samples. $y^*$ are obtained by solving the inner maximization problem in \eqref{problem:SAA NDRO} for the DRO decision $x^*$, which is the final iterate $x^K$.
As depicted in Figure \ref{fig:NDRO_loss}, the test loss exhibits a rapid decrease, underscoring the effectiveness of the isHI method for this class of nonlinear DRO problems.

To provide insight into how the distribution shifts from the initial empirical samples towards the worst-case configuration, Figure \ref{fig:NDRO_sample} shows the trajectories of the sample points $\xi_i$. It is observable that all sample points migrate closer to the origin. Notably, points that were initially more distant from the origin exhibit larger displacements inward. This observed behavior is consistent with the structure of $\ell(x,\xi)$, where deviations towards the origin can be more impactful due to the interplay between the quadratic and exponential terms involving $\xi$. 

In summary, the results demonstrate that the inexact stochastic Halpern iteration can efficiently solve DRO problems with convex-concave loss function even when dual methods are inapplicable. These insights suggest promising extensions to broader classes of robust optimization problems.

\begin{figure}[htb!]
    \centering

    \begin{subfigure}[t]{0.48\textwidth}
        \centering
        \includegraphics[width=0.9\linewidth]{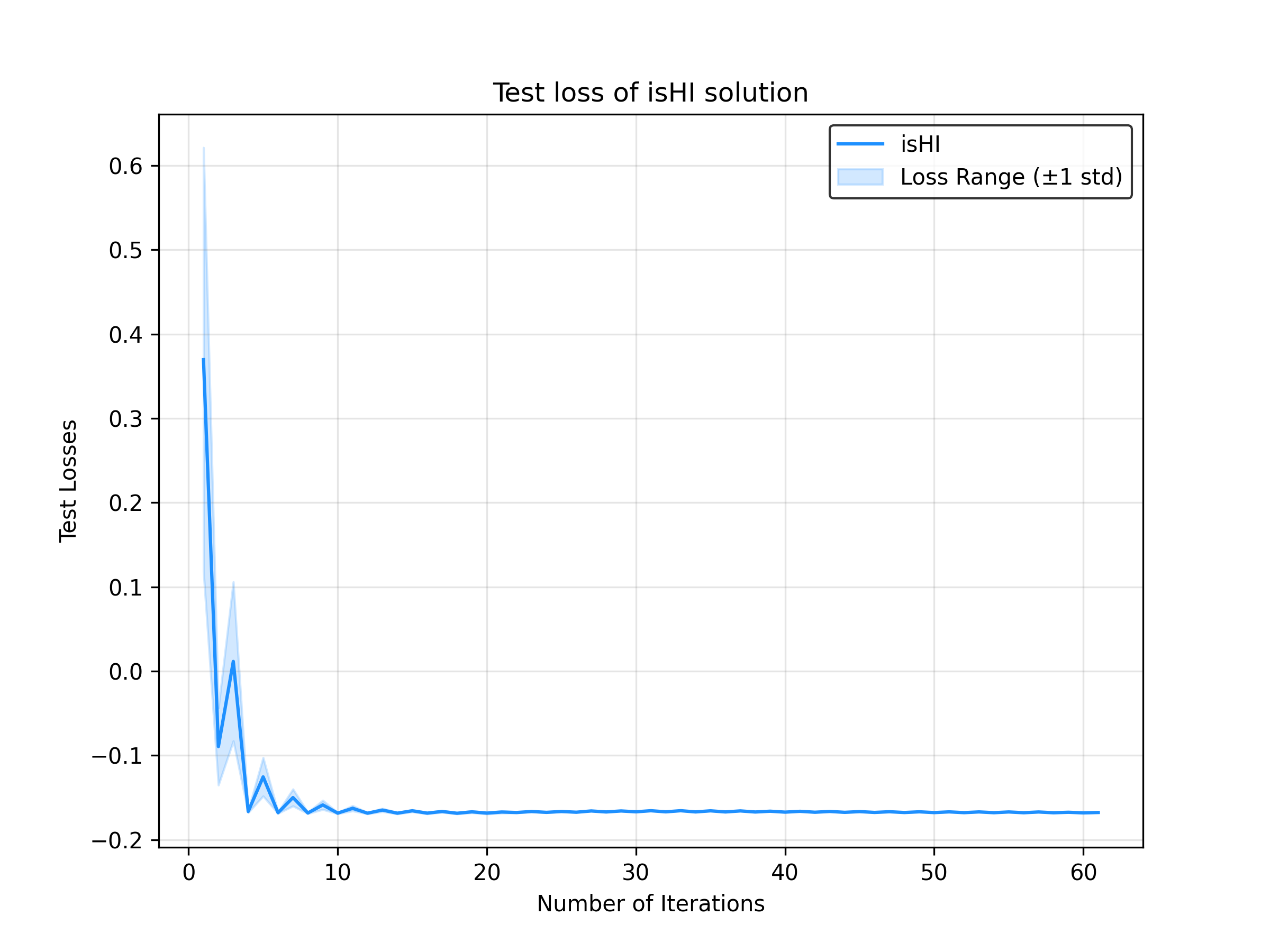}
        \caption{The blue curve represents the mean test losses over 10 runs, with the shaded region showing $\pm 1$ standard deviation.}
        \label{fig:NDRO_loss}
    \end{subfigure}
    \hfill
    \begin{subfigure}[t]{0.48\textwidth}
        \centering
        \includegraphics[width=0.9\linewidth]{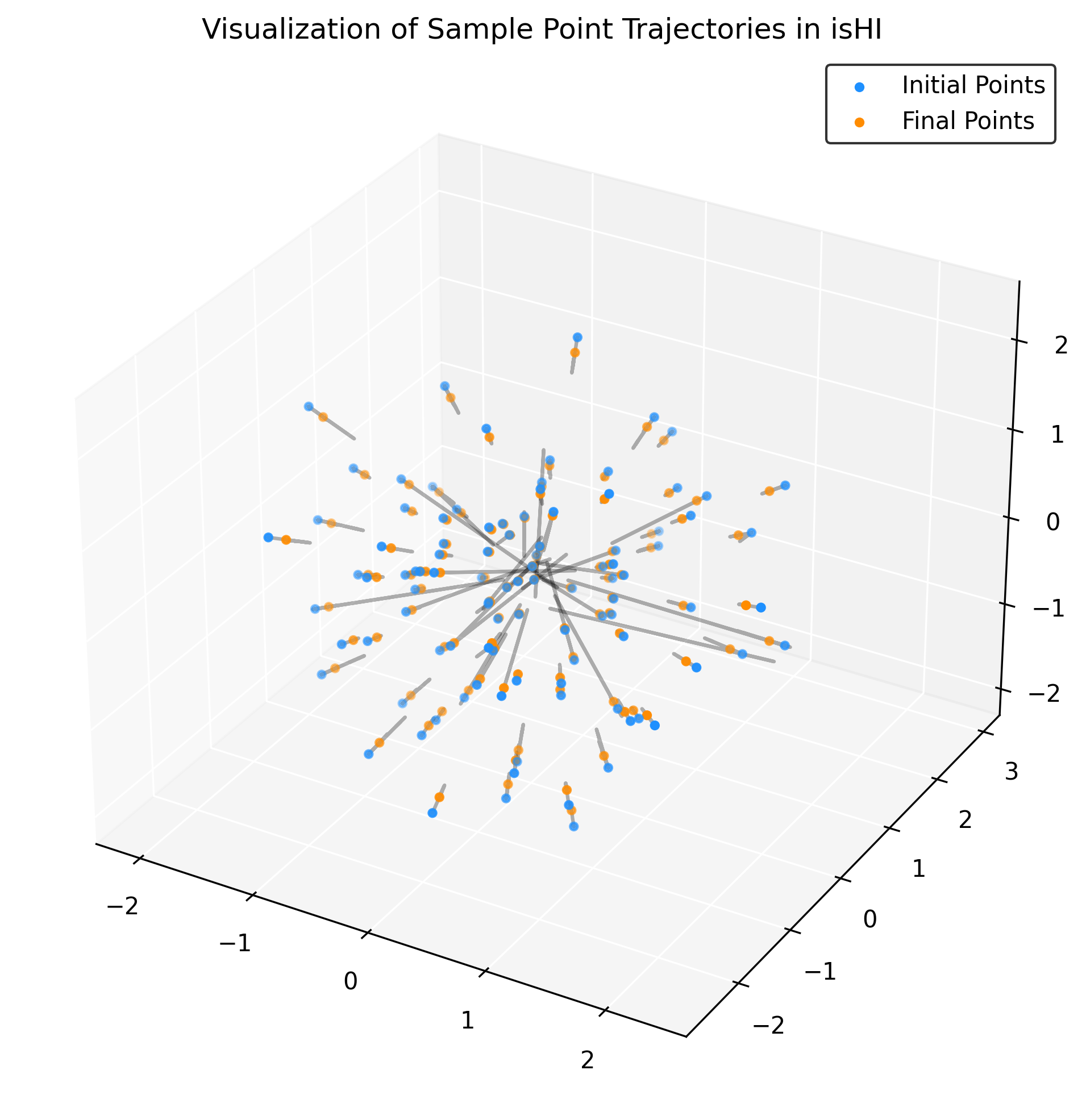}
        \caption{The blue points represent the initial empirical samples. The orange points represent the worst-case perturbed points upon termination. The gray lines depict the trajectories of the sample points under the isHI updates.}
        \label{fig:NDRO_sample}
    \end{subfigure}

    \caption{Visualization of test loss dynamics and sample perturbation under the isHI updates.}
    \label{fig:NDRO_combined}
\end{figure}

\section{Conclusions}
\label{sec:conclusions}
We have analyzed two inexact versions of the classical Halpern fixed-point iterative scheme via conducting comprehensive convergence analysis. In particular, we have established the $O(k^{-1})$ convergence rates in terms of the residue norm and expected residue norm in deterministic and stochastic settings, respectively. We adapt the proposed methods for solving two important classes of data-driven Wasserstein distributionally robust optimization (WDRO) problems that can be reformulated as convex-concave min-max optimization problems. However, in many real world applications, the WDRO may not guarantee convex-concave min-max or convex optimization reformulations. In this case, the Halpern iteration and other related algorithms designed for convex and/or convex-concave min-max optimization are no longer applicable.

\section*{Acknowledgement}

The research of Jia-Jie Zhu received funding support from the Deutsche Forschungsgemeinschaf (DFG, German Research Foundation) as part of the priority programme ``Theoretical foundations of deep learning" (project number: 543963649). We thank the editor and the anonymous reviewers for providing  valuable suggestions, which have helped to improve the quality of the paper.

\bibliographystyle{spmpsci}
\bibliography{references}  

\appendix

\section{Technical lemmas and proofs.}

We first show in the following lemma some useful expressions for the difference $z^{k+1} - z^k$ in the Halpern iteration \eqref{eq-inexact-halpern}.
\begin{lemma}
    \label{lemma-zdiff}
    Let $\{z^k\}$ be generated by \eqref{eq-inexact-halpern} with $\eta_k = (1-\beta_k)/L$, then the following expressions for $z^{k+1} - z^k$ for $k\geq 0$ hold:
    \begin{enumerate}
        \item $z^{k+1} - z^k = \beta_k(z^0 - z^k) - \frac{1-\beta_k}{L}\tilde{z}^k$;
        \item $z^{k+1} - z^k = \frac{\beta_k}{1-\beta_k}(z^0 - z^{k+1}) - \frac{1}{L}\tilde{z}^k$;
        \item $z^{k+1} - z^k = -\frac{1-\beta_k}{L}\tilde{z}^k + \sum_{i = 0}^{k-1}\frac{\beta_k}{L}\left( \prod_{j = i}^{k-1} (1-\beta_j)\right) \tilde{z}^i$, for $k\geq 0$. Moreover, if $\beta_k:= 1/(k+2)$, it holds that 
    \begin{equation}
        \label{eq-lemma-zdiff-1}
        \begin{aligned}
            &\; \norm{z^{k+1} - z^k}^2 \\
            \leq &\; \left\{
            \begin{array}{ll}
                \frac{1}{4L^2} \norm{\tilde{z}^0}^2, & \textrm{if } k = 0,\\[5pt]
                \frac{2(k+1)^2}{L^2(k+2)^2}\norm{\tilde{z}^k}^2 + \frac{2k}{L^2(k+1)^2(k+2)^2}\sum_{i = 0}^{k-1}(i+1)^2\norm{\tilde{z}^i}^2,    & \textrm{if } k\geq 1. 
            \end{array}
        \right.
        \end{aligned}
    \end{equation}
    \end{enumerate}
\end{lemma}

\begin{proof}
    The first statement of the lemma is obvious. For the second statement, we first note that 
    \[
        z^0 - z^{k+1} = (1 - \beta_k)(z^0 - z^k) + \frac{1-\beta_k}{L}\tilde{z}^k.
    \]
    Hence, it holds that
    \[
        z^0 - z^k = \frac{1}{1 - \beta_k} (z^0 - z^{k+1}) - \frac{1}{L} \tilde{z}^k,
    \]
    which together with part 1 further implies that 
    \[
        z^{k+1} - z^k = \beta_k(z^0 - z^k) - \frac{1-\beta_k}{L}\tilde{z}^k = \frac{\beta_k}{1 - \beta_k} (z^0 - z^{k+1}) - \frac{1}{L} \tilde{z}^k.
    \]
    Therefore, the second statement of the lemma holds true. We next show the third statement via mathematical induction. For $k = 0$, we see from the first statement of the lemma that 
    \[
        z^1 - z^0 = -\frac{{1-\beta_0}}{L}\tilde{z}^0,
    \]
    which implies that the third statement of the lemma holds for $k = 0$. Now suppose that the third statement of the lemma holds for $k-1$, i.e., 
    \[
        z^{k} - z^{k-1} = -\frac{1-\beta_{k-1}}{L}\tilde{z}^{k-1} + \sum_{i = 0}^{k-2}\frac{\beta_{k-1}}{L}\left( \prod_{j = i}^{k-2} (1-\beta_j)\right) \tilde{z}^i.
    \]
    Then, we can see that
        \begin{align*}
             &\; z^{k+1} - z^k \\
            = &\;  - \frac{1-\beta_k}{L}\tilde{z}^k + \beta_k(z^0 - z^k) \\
            = &\; - \frac{1-\beta_k}{L}\tilde{z}^k + \beta_k\left( \frac{1-\beta_{k-1}}{\beta_{k-1}}(z^k - z^{k-1}) + \frac{1-\beta_{k-1}}{L\beta_{k-1}}\tilde{z}^{k-1}  \right) \\
             = &\; - \frac{1-\beta_k}{L}\tilde{z}^k +  \frac{\beta_k(1-\beta_{k-1})}{\beta_{k-1}}\left( -\frac{1-\beta_{k-1}}{L}\tilde{z}^{k-1} + \sum_{i = 0}^{k-2}\frac{\beta_{k-1}}{L}\left( \prod_{j = i}^{k-2} (1-\beta_j)\right) \tilde{z}^i\right) \\
             &\; + \frac{\beta_k(1-\beta_{k-1})}{L\beta_{k-1}}\tilde{z}^{k-1}  \\
             =&\; - \frac{1-\beta_k}{L}\tilde{z}^k + \sum_{i = 0}^{k-2}\frac{\beta_k(1-\beta_{k-1})}{L}\left( \prod_{j = i}^{k-2} (1-\beta_j)\right) \tilde{z}^i + \frac{\beta_{k}(1 - \beta_{k-1})}{L}\tilde{z}^{k-1} \\
             =&\; -\frac{1-\beta_k}{L}\tilde{z}^k + \sum_{i = 0}^{k-1}\frac{\beta_k}{L}\left( \prod_{j = i}^{k-1} (1-\beta_j)\right) \tilde{z}^i.
        \end{align*}
    Here, we use the first statement of the lemma in the first equality and use the second statement of the lemma in the second equality. Moreover, the third equality is due to the induction assumption and the remaining inequalities are derived from some direct simplifications. Therefore, the statement is true for $k$. By induction, we see that it is true for all $k\geq 0$.
    
    Finally, if $\beta_k = 1/(k+2)$, we have that $\beta_0 = 1/2$ and 
    \[
        z^1 - z^0 = \frac{1}{2L}\tilde{z}^0,
    \]
    which implies that 
    \[
        \norm{z^1 - z^0}^2 = {\frac{1}{4L^2}}\norm{\tilde{z}^0}^2.
    \]
    For $k\geq 1$, substituting $\beta_k = 1/(k+2)$ into the  expression of $z^{k+1} - z^k$ in the third statement yields that 
    \begin{align*}
        z^{k+1} - z^k = &\; -\frac{k+1}{L(k+2)}\tilde{z}^k + \sum_{i = 0}^{k-1}\frac{1}{L(k+2)}\left( \prod_{j = i}^{k-1} \frac{j+1}{j+2}\right) \tilde{z}^i \\
        =&\;  -\frac{k+1}{L(k+2)}\tilde{z}^k + \sum_{i = 0}^{k-1}\frac{i+1}{L(k+1)(k+2)}\tilde{z}^i.
    \end{align*}
    Therefore, by applying the Cauchy-Schwarz inequality twice, we see that 
        \begin{align*}
            \norm{z^{k+1} - z^k}^2 = &\; \norm{-\frac{k+1}{L(k+2)}\tilde{z}^k + \sum_{i = 0}^{k-1}\frac{i+1}{L(k+1)(k+2)}\tilde{z}^i}^2 \\
            \leq &\; \frac{2(k+1)^2}{L^2(k+2)^2}\norm{\tilde{z}^k}^2 + \frac{2}{L^2(k+1)^2(k+2)^2}\norm{\sum_{i = 0}^{k-1}(i+1)\tilde{z}^i}^2 \\
            \leq &\; \frac{2(k+1)^2}{L^2(k+2)^2}\norm{\tilde{z}^k}^2 + \frac{2k}{L^2(k+1)^2(k+2)^2}\sum_{i = 0}^{k-1}(i+1)^2\norm{\tilde{z}^i}^2,
        \end{align*}    
    which proves \eqref{eq-lemma-zdiff-1}. Hence, the proof is completed.
\end{proof}

Using the above lemma and the co-coerciveness of the mapping $G$, we can prove the following key estimate that is crucial for analyzing the rate of the convergence for \eqref{eq-inexact-halpern}. 
\begin{lemma}
    \label{lemma-key-estimate}
    Let $\{z^k\}$ be generated by \eqref{eq-inexact-halpern} with $\eta_k = (1-\beta_k)/L$, then for $k\geq 0$, it holds that 
    \begin{equation}
        \label{eq-lemma-key-estimate-1}
        \begin{aligned}
            &\; \frac{1}{2L}\norm{G(z^{k+1})}^2 - \frac{\beta_k}{1-\beta_k}\inprod{G(z^{k+1})}{z^0 - z^{k+1}} \\ \leq &\;   \frac{1 - 2\beta_k}{2L}\norm{G(z^k)}^2 - \beta_k\inprod{G(z^k)}{z^0 - z^k} + \frac{\beta_k}{L}\inprod{G(z^k)}{G(z^k) - \tilde{z}^k}\\
            &\; + \frac{1}{2L}\norm{G(z^k) - \tilde{z}^k}^2.
        \end{aligned}
    \end{equation}
    In particular, if $\beta_k:=1/(k+2)$, then for $k\geq 0$,
    \begin{equation}
        \label{eq-lemma-key-estimate-2}
        \mathcal{L}_{k+1} \leq \mathcal{L}_k  + \frac{1}{12L}\norm{G(z^k)}^2 + \frac{4(k+1)^2}{L}\norm{G(z^k) - \tilde{z}^k}^2,
    \end{equation}
    where $\mathcal{L}_k$ is the potential function that is defined as 
    \[
        \mathcal{L}_k:= \frac{k(k+1)}{2L}\norm{G(z^k)}^2 - (k+1)\inprod{G(z^k)}{z^0 - z^k},\quad \forall k\geq 0.
    \]
\end{lemma}

\begin{proof}
    By the co-coerciveness of the mapping $G$, we see that 
    \[
        \frac{1}{L}\norm{G(z^{k+1}) - G(z^k)}^2 \leq \inprod{G(z^{k+1}) - G(z^k)}{z^{k+1} - z^k},
    \]  
    which implies that 
    \begin{equation}
        \label{eq-lemma-key-estimate-3}
        \begin{aligned}
            &\; \frac{1}{2L}\norm{G(z^{k+1})}^2 + \frac{1}{2L}\norm{G(z^{k})}^2 +  \frac{1}{2L}\norm{G(z^{k+1}) - G(z^k)}^2 \\
            \leq &\;  \inprod{G(z^{k+1})}{z^{k+1} - z^k + \frac{1}{L}G(z^k)} - \inprod{G(z^k)}{z^{k+1} - z^k}.
        \end{aligned}
    \end{equation}
    The first two expressions for $z^{k+1} - z^k$ in Lemma \ref{lemma-zdiff} together with \eqref{eq-lemma-key-estimate-3} and the Cauchy-Schwarz inequality imply that 
        \begin{align*}
            &\; \frac{1}{2L}\norm{G(z^{k+1})}^2 + \frac{1}{2L}\norm{G(z^{k})}^2 +  \frac{1}{2L}\norm{G(z^{k+1}) - G(z^k)}^2 \\
            \leq &\; \inprod{G(z^{k+1})}{\frac{\beta_k}{1-\beta_k}(z^0 - z^{k+1}) - \frac{1}{L}\tilde{z}^k + \frac{1}{L}G(z^k)} - \inprod{G(z^k)}{\beta_k(z^0 - z^k) - \frac{1-\beta_k}{L}\tilde{z}^k} \\
            = &\; \frac{\beta_k}{1-\beta_k}\inprod{G(z^{k+1})}{z^0 - z^{k+1}} - \beta_k\inprod{G(z^k)}{z^0 - z^k} + \frac{1}{L}\inprod{G(z^{k+1})}{G(z^k) - \tilde{z}^k} \\
            &\; + \frac{1-\beta_k}{L}\inprod{G(z^k)}{\tilde{z}^k}\\
            = &\; \frac{\beta_k}{1-\beta_k}\inprod{G(z^{k+1})}{z^0 - z^{k+1}} - \beta_k\inprod{G(z^k)}{z^0 - z^k}  + \frac{1}{L}\norm{G(z^k)}^2 \\
            &\; + \frac{1}{L}\inprod{G(z^{k+1}) - G(z^k)}{G(z^k) - \tilde{z}^k} - \frac{\beta_k}{L}\inprod{G(z^k)}{\tilde{z}^k} \\
            \leq &\; \frac{\beta_k}{1-\beta_k}\inprod{G(z^{k+1})}{z^0 - z^{k+1}} - \beta_k\inprod{G(z^k)}{z^0 - z^k}  + \frac{1 - \beta_k}{L}\norm{G(z^k)}^2 \\
            &\; + \frac{1}{2L}\norm{G(z^{k+1}) - G(z^k)}^2 + \frac{1}{2L}\norm{G(z^k) - \tilde{z}^k}^2 + \frac{\beta_k}{L}\inprod{G(z^k)}{G(z^k) - \tilde{z}^k}.
        \end{align*}
    Therefore, the inequality \eqref{eq-lemma-key-estimate-1} holds by rearranging terms in the above inequality. By substituting $ \beta_k:= 1 / (k+2) $ into \eqref{eq-lemma-key-estimate-1}, and using the Cauchy-Schwarz inequality and the definition of the potential function $ \mathcal{L}_k $, for $k\geq 0$, we see that 
    	\begin{align*}
    		\mathcal{L}_{k+1} \leq &\;  \mathcal{L}_k  + \frac{k+1}{L}\inprod{G(z^k)}{G(z^k) - \tilde{z}^k} + \frac{(k+1)(k+2)}{2L}\norm{G(z^k) - \tilde{z}^k}^2 \\
    		\leq &\; \mathcal{L}_k  + \frac{k+1}{L}\inprod{G(z^k)}{G(z^k) - \tilde{z}^k} + \frac{(k+1)^2}{L}\norm{G(z^k) - \tilde{z}^k}^2 \\
    		\leq &\; \mathcal{L}_k  + \frac{k+1}{2L}\left(\frac{1}{6(k+1)}\norm{G(z^k)}^2 + 6(k+1) \norm{G(z^k) - \tilde{z}^k}^2 \right) \\
            &\; + \frac{(k+1)^2}{L}\norm{G(z^k) - \tilde{z}^k}^2 \\
    		= &\; \mathcal{L}_k  + \frac{1}{12L}\norm{G(z^k)}^2 + \frac{4(k+1)^2}{L}\norm{G(z^k) - \tilde{z}^k}^2,
    	\end{align*}
    which proves the inequality \eqref{eq-lemma-key-estimate-2}. Therefore, the proof is completed.
\end{proof}

The following lemma quantifies the inexactness in the expectation $\bb{E} \left[\norm{G(z^k) - \tilde{z}^k}\right]$, which is key to our later convergence analysis. Here, we use $\bb{E}$ to denote the expectation with respect to all the randomness at any iteration of \eqref{eq-isHalpern}.
\begin{lemma}
	\label{lemma-exp-error}
	Let $\{z^k\}$ be the sequence generated by \eqref{eq-isHalpern}, then it holds that
	\[
		\bb{E} \left[\norm{G(z^k) - \tilde{z}^k}^2\right] \leq \gamma_k + \sigma_k,\quad \forall k\geq 0.
	\]
\end{lemma}  

\begin{proof}
	Since the resolvent operator $J_{\alpha E}$ is firmly non-expansive (hence $1$-Lipschitz continuous), we see that 
		\begin{align*}
			&\; \norm{G(z^k) - \tilde{z}^k}^2 \\
			= &\; \norm{\frac{1}{\alpha}\left(z^k - J_{\alpha E}\left(z^k -  \alpha F(z^k)\right)\right)- \tilde{z}^k}^2 \\
			= &\; \frac{1}{\alpha^2}\norm{J_{\alpha E}\left(z^k -  \alpha F(z^k)\right) - \bar{z}^k}^2 \\
			\leq &\; \frac{2}{\alpha^2}\norm{J_{\alpha E}\left(z^k -  \alpha F(z^k)\right) - J_{\alpha E}\left(z^k -  \alpha \widetilde{F}(z^k)\right)}^2 + \frac{2}{\alpha^2}\norm{J_{\alpha E}\left(z^k - \alpha \widetilde{F}(z^k)\right) - \bar{z}^k}^2 \\
			\leq &\; 2\norm{F(z^k) - \widetilde{F}(z^k)}^2 + \gamma_k,
		\end{align*}
    where the first equality uses the definition of $G$ and the second equality uses the fact that $\tilde z^k = \frac{1}{\alpha}(z^k - \bar{z}^k)$, and the first inequality is due to the Cauchy-Schwarz inequality.
	Then, it follows that: 
	\[
		\bb{E} \left[\norm{G(z^k) - \tilde{z}^k}^2\right] \leq  2\bb{E}\left[\norm{F(z^k) - \widetilde{F}(z^k)}^2\right] +  \gamma_k 
			\leq    \sigma_k + \gamma_k,
	\]
	which completes the proof.
\end{proof}

The next lemma shows that $f_i$ defined in \eqref{eq-WDRSL-minmax} are L-smooth.
\begin{lemma}
    \label{lemma-fi-Lip}
    Let  $f_i:\bb{R}^d\times\bb{R}^N\to \bb{R}$ be defined in \eqref{eq-WDRSL-minmax}, i.e., for $x := (w,\lambda)\in \bb{R}^d$ and $y\in \bb{R}^N$,
    \[
        f_i(x,y):=\Psi_0(w) + \lambda(\theta-\kappa) + \Psi\left(\inprod{\hat\phi_i}{w}\right) + y_i\left(\hat\psi_i\inprod{\hat\phi_i}{w} - \lambda\kappa\right),\quad \forall\;  i = 1,\dots,N.
    \]
    Then, there exists a constant $L_0 >0$ such that 
    \[
        \norm{\nabla f_{i}(x', y') - \nabla f_{i}(x, y)} \leq L_0 \norm{\begin{pmatrix}
		  	x' - x \\ y' - y
		  \end{pmatrix}},\quad \forall\; x,x' \in \bb{R}^d,\; \forall\; y,y'\in \bb{R}^N,\; \forall \;i = 1,\dots,N.
    \]
\end{lemma}

\begin{proof}
    We consider any $x := (w,\lambda), \; x':=(w',\lambda')\in \bb{R}^d$, and $y,y'\in\bb{R}^N$. A simple calculation shows that 
	\begin{align*}
		&\; \norm{\nabla f_{i}(x', y') - \nabla f_{i}(x, y)}^2 \\
		 = &\; \norm{\Psi_0'(w') - \Psi_0'(w) + \Psi'\left(\inprod{\hat\phi_i}{w'}\right)\hat\phi_i - \Psi'\left(\inprod{\hat\phi_i}{w}\right)\hat\phi_i + (y_i' - y_i) \hat\psi_i \hat\phi_i}^2 \\
		 &\;  + \kappa^2\abs{y_i' - y_i}^2 + \norm{\left(\hat\psi_i\inprod{\hat\phi_i}{w'} - \lambda' \kappa\right) e_i - \left(\hat\psi_i\inprod{\hat\phi_i}{w} - \lambda \kappa\right) e_i}^2 \\
		 \leq &\; 3\norm{\Psi_0'(w') - \Psi_0'(w)}^2 + 3\norm{\Psi'\left(\inprod{\hat\phi_i}{w'}\right)\hat\phi_i - \Psi'\left(\inprod{\hat\phi_i}{w}\right)\hat\phi_i}^2 + 3\norm{(y_i' - y_i) \hat\psi_i \hat\phi_i}^2 \\
		  &\;  + \kappa^2\abs{y_i' - y_i}^2 +2\norm{\left(\hat\psi_i\inprod{\hat\phi_i}{w'} -  \hat\psi_i\inprod{\hat\phi_i}{w}\right) e_i}^2 + 2\kappa^2\abs{\lambda' - \lambda}^2 \\
		  \leq &\; 3\widetilde{L}_0^2\norm{w' - w}^2 +  3\norm{\hat\phi_i}^2\abs{\Psi'\left(\inprod{\hat\phi_i}{w'}\right) - \Psi'\left(\inprod{\hat\phi_i}{w}\right)}^2 +  \left(3\norm{\hat\phi_i}^2+\kappa^2\right)\abs{y_i' - y_i}^2 \\
		  &\;    +  2\norm{\hat\phi_i}^2\norm{w' -  w}^2 + 2\kappa^2\abs{\lambda' - \lambda}^2 \\
		  \leq &\; \left(3\widetilde{L}_0^2 + 3\widetilde{L}_0^2\norm{\hat\phi_i}^4+2\norm{\hat\phi_i}^2\right)\norm{w'-w}^2 +  \left(3\norm{\hat\phi_i}^2+\kappa^2\right)\abs{y_i' - y_i}^2 + 2\kappa^2\abs{\lambda' - \lambda}^2 \\
		  \leq &\; \max\left\{3\widetilde{L}_0^2 + 3\widetilde{L}_0^2\left(\max_{1\leq i\leq N}\norm{\hat\phi_i}\right)^4+2\left(\max_{1\leq i\leq N}\norm{\hat\phi_i}\right)^2, 2\kappa^2 \right\}\norm{x' - x}^2 \\
		  &\;  + \left(3\left(\max_{1\leq i\leq N}\norm{\hat\phi_i}\right)^2+\kappa^2\right)\norm{y' - y}^2 \\
		  \leq &\;\underbrace{\max\left\{3\widetilde{L}_0^2 + 3\widetilde{L}_0^2\left(\max_{1\leq i\leq N}\norm{\hat\phi_i}\right)^4+2\left(\max_{1\leq i\leq N}\norm{\hat\phi_i}\right)^2, 2\kappa^2, \left(3\left(\max_{1\leq i\leq N}\norm{\hat\phi_i}\right)^2+\kappa^2\right) \right\}}_{L_0^2}\\
    &\; \cdot  \norm{\begin{pmatrix}
		  	x' - x \\ y' - y
		  \end{pmatrix}}^2,
	\end{align*}
where the first three inequalities are derived from the Cauchy-Schwarz inequality and the fact that $\Psi_0$ and $\Psi$ have Lipschitz gradients and $\hat\psi_i\in\{-1,1\}$, for $i = 1,\dots, N$. The above inequality further shows that $f_{i}$ is $L_0$-smooth, for $i = 1,\dots, N$. Thus, the proof is completed.
\end{proof}

Under a certain growth condition for the loss function $\ell$, known results in the literature have shown that the inner problem $\sup_{\mu\in \bb{B}_\theta(\hat{\mu}_N)}\; \bb{E}_{\xi\sim \mu}\left[\ell(x, \xi)\right]$ in the WDRO problem \eqref{eq-wdro} has a strong dual problem. These results are summarized in the following lemma.
\begin{lemma}{\cite{gao2022distributionally}}
	\label{lemma-polish-dual}
	Suppose that $\Xi$ is convex and for any $ x\in \bb{R}^n $,
	\[
		\limsup_{d(\xi,\xi_0)\to \infty}\frac{\ell(x, \xi) - \ell(x, \xi_0)}{d^p(\xi, \xi_0)} < \infty,
	\]
	where $\xi_0\in \Xi$ is any given point. Then, for any $x\in \bb{R}^n$,
	\[
		\sup_{\mu\in \bb{B}_\theta(\hat{\mu}_N)} \bb{E}_{\xi\sim \mu}\left[\ell(x, \xi)\right] = \sup_{\xi_i\in \Xi, i = 1,\dots, N}\left\{ \frac{1}{N}\sum_{i = 1}^N\ell(x,\xi_i): \frac{1}{N}\sum_{i=1}^Nd^p(\xi_i, \hat{\xi}_i) \leq \theta^p\right\},
	\]
	where $\{\hat{\xi}_i\}_{i = 1}^N$ denotes the training data set, $\hat{\mu}_N$ denotes the empirical distribution with respect to the training data set and $\theta > 0$ is a given parameter.
\end{lemma}

\section{Proofs of Main Results}
\subsection{Proof of Theorem \ref{thm-convergence-rate}}\label{proof-thm-rate}

\begin{proof}
	From Lemma \ref{lemma-key-estimate} and notice that $\mathcal{L}_0 = 0$, we can see that  
	\[
		\mathcal{L}_k \leq \frac{1}{12L}\sum_{i = 0}^{k-1}\norm{G(z^i)}^2 + \frac{4}{L}\sum_{i = 0}^{k-1}(i+1)^2\gamma_i^2, \quad k\geq 1.
	\]
	By the definition of the potential function $\mathcal{L}_k$, we can see that 
		\begin{align*}
			\frac{k(k+1)}{2L}\norm{G(z^k)}^2 \leq &\; \frac{1}{12L}\sum_{i = 0}^{k-1}\norm{G(z^i)}^2 + \frac{4}{L}\sum_{i = 0}^{k-1}(i+1)^2\gamma_i^2 + (k+1)\inprod{G(z^k)}{z^0 - z^k} \\
			\leq &\; \frac{1}{12L}\sum_{i = 0}^{k-1}\norm{G(z^i)}^2 + \frac{4}{L}\sum_{i = 0}^{k-1}(i+1)^2\gamma_i^2 + (k+1)\norm{G(z^k)}\norm{z^0 - z^*}
		\end{align*} 
	where the last inequality is due to the fact that 
	\begin{align*}
	    \inprod{G(z^k)}{z^0 - z^k} = &\; \inprod{G(z^k)}{z^0 - z^*} + \inprod{G(z^k)}{z^* - z^k} \\
     \leq &\;  \inprod{G(z^k)}{z^0 - z^*} \\
     \leq &\;  \norm{G(z^k)}\norm{z^0 - z^*},
	\end{align*}
	and $z^*$ is any solution such that $G(z^*) = 0$.
 The first inequality holds, since $G$ is co-coercive. As a consequence, it holds that 
	\begin{equation}
	\label{thm-rate-proof-1}
		\norm{G(z^k)}^2 \leq \frac{1}{6k(k+1)}\sum_{i = 0}^{k-1}\norm{G(z^i)}^2 + \frac{8}{k(k+1)}\sum_{i = 0}^{k-1}(i+1)^2\gamma_i^2 + \frac{2L}{k}\norm{G(z^k)}\norm{z^0 - z^*},  
	\end{equation}
	for $k\geq 1$. For notational simplicity, let us denote 
	\[
		a_k:= \norm{G(z^k)}, \quad k\geq 0, \quad b_k:= \sum_{i = 0}^{k-1}(i+1)^2\gamma_i^2,\quad k\geq 1,\quad D:= L\norm{z^0 - z^*}.
	\]
	Then, from \eqref{thm-rate-proof-1}, we see that, for $k\geq 1$, 
	\begin{align*}
		a_k^2 \leq &\;  \frac{1}{6k(k+1)}\sum_{i = 0}^{k-1} a_i^2 + \frac{8}{k(k+1)}b_k + \frac{2D}{k}a_k \\
		\leq &\; \frac{1}{6k(k+1)}\sum_{i = 0}^{k-1} a_i^2 + \frac{8}{k(k+1)}b_k + \frac{1}{2}\left(\frac{4D^2}{k^2} + a_k^2 \right),
	\end{align*}
	which implies that, for $k\geq 1$,
	\begin{equation}
	\label{thm-rate-proof-2}
	\begin{aligned}
	    a_k^2 \leq &\; \frac{1}{3k(k+1)}\sum_{i = 0}^{k-1} a_i^2 + \frac{16}{k(k+1)}b_k + \frac{4D^2}{k^2} \\
     \leq &\; \frac{1}{3k(k+1)}\sum_{i = 0}^{k-1} a_i^2 + \frac{48}{(k+1)(k+2)}b_k + \frac{24D^2}{(k+1)(k+2)}.
	\end{aligned}
	\end{equation}
	We next claim that the following inequality holds:
	\begin{equation}
		\label{thm-rate-proof-3}
		a_k^2 \leq \frac{(7D + 10\sqrt{b_k})^2}{(k+1)(k+2)},\quad k\geq 1.
	\end{equation}
	We shall prove the above claim by induction. First, from \eqref{thm-rate-proof-2} and notice that $a_0 \leq D$ by the $L$-Lipschitz continuity of $G$, we see that
	\[
		a_1^2 \leq \frac{1}{6}a_0^2 + 8b_1 + 4D^2 \leq  \frac{1}{6}D^2 + 8b_1 + 4D^2 \leq \frac{1}{6}(7D+10\sqrt{b_1})^2.
	\]
	So \eqref{thm-rate-proof-3} holds for $k = 1$. Suppose now that \eqref{thm-rate-proof-3} holds for indices $\{1,\dots, k-1\}$. Again, from \eqref{thm-rate-proof-2} and the fact that $b_k$ is non-decreasing, we deduce that 
		\begin{align*}
			a_k^2 \leq &\; \frac{1}{3k(k+1)}\sum_{i = 0}^{k-1} a_i^2 + \frac{48}{(k+1)(k+2)}b_k + \frac{24D^2}{(k+1)(k+2)} \\
			\leq &\; \frac{1}{3k(k+1)}\sum_{i = 0}^{k-1} \frac{(7D + 10\sqrt{b_i})^2}{(i+1)(i+2)} + \frac{48}{(k+1)(k+2)}b_k + \frac{24D^2}{(k+1)(k+2)} \\
			\leq &\; \frac{(7D + 10\sqrt{b_k})^2}{3k(k+1)}\sum_{i = 0}^{k-1} \frac{1}{(i+1)(i+2)} + \frac{48}{(k+1)(k+2)}b_k + \frac{24D^2}{(k+1)(k+2)} \\
			= &\; \frac{(7D + 10\sqrt{b_k})^2}{3(k+1)^2} + \frac{48}{(k+1)(k+2)}b_k + \frac{24D^2}{(k+1)(k+2)} \\
			\leq &\; \frac{(7D + 10\sqrt{b_k})^2}{(k+1)(k+2)},
		\end{align*}
	which shows that \eqref{thm-rate-proof-3} holds for index $k$. Thus, by induction, \eqref{thm-rate-proof-3} holds for all $k\geq 1$. Moreover, it is clear that \eqref{thm-rate-proof-3} trivially holds for $k = 0$. This proves \eqref{thm-rate-1}.
	
	To prove  \eqref{thm-rate-2}, we first observe that 
	\[
		\norm{\tilde{z}^k}^2 \leq 2\norm{G(z^k)}^2 + 2\norm{G(z^k) - \tilde{z}^k}^2 \leq 2a_k^2 + 2\gamma_k^2 \leq  \frac{2(7D + 10\sqrt{b_k})^2}{(k+1)(k+2)} + 2\gamma_k^2,\quad k\geq 0.
	\]
	Then, by the third statement of Lemma \ref{lemma-zdiff} and the fact that $b_k$ is non-decreasing, we see that for $k\geq 0$, 
		\begin{align*}
			\norm{z^{k+1} - z^k}^2
			\leq &\; \frac{2(k+1)^2}{L^2(k+2)^2}\norm{\tilde{z}^k}^2 + \frac{2k}{L^2(k+1)^2(k+2)^2}\sum_{i = 0}^{k-1}(i+1)^2\norm{\tilde{z}^i}^2 \\
			\leq &\;  \frac{4(k+1)^2}{L^2(k+2)^2}\left(\frac{(7D + 10\sqrt{b_k})^2}{(k+1)(k+2)} + \gamma_k^2 \right) \\
   &\; + \frac{4k}{L^2(k+1)^2(k+2)^2}\sum_{i = 0}^{k-1}\left(\frac{(i+1)(7D + 10\sqrt{b_k})^2}{i+2} + (i+1)^2 \gamma_i^2 \right) \\
			\leq &\; \frac{4(7D + 10\sqrt{b_k})^2}{L^2}\left(\frac{k+1}{(k+2)^3} + \frac{k^2}{(k+1)^2(k+2)^2}\right) \\
   &\; + \frac{4(k+1)^2}{L^2(k+2)^2}\gamma_k^2 + \frac{4k}{L^2(k+1)^2(k+2)^2}b_k \\
			\leq &\; \frac{8(7D + 10\sqrt{b_{k+1}})^2 + 8b_{k+1}}{L^2(k+1)(k+2)}\\
			\leq &\; \frac{8(7D + 11\sqrt{b_{k+1}})^2 }{L^2(k+1)(k+2)},
		\end{align*}
	which proves \eqref{thm-rate-2}. Therefore, the proof is completed.
\end{proof}

\subsection{Proof of Theorem \ref{thm-rate-ishi}}\label{proof-thm-rate-ishi}

\begin{proof}
	We recall from Lemma \ref{lemma-key-estimate} that 
	\[
		\mathcal{L}_{k+1} \leq \mathcal{L}_k  + \frac{1}{12L}\norm{G(z^k)}^2 + \frac{4(k+1)^2}{L}\norm{G(z^k) - \tilde{z}^k}^2,\quad k\geq 0.
	\]
	Taking the full expectation $\bb{E}$ (with respect to all the randomness) on both side of the above inequality yields 
	\[
		\bb{E}[\mathcal{L}_{k+1}] \leq \bb{E}[\mathcal{L}_k]  + \frac{1}{12L}\bb{E}\left[\norm{G(z^k)}^2\right] + \frac{4(k+1)^2}{L}(\sigma_k + \gamma_k)^2,\quad k\geq 0,
	\]
	where the third term is due to Lemma \ref{lemma-exp-error}. Then, by recursively applying the above and summing up the resulted inequalities gives:
	\begin{equation}\label{eq-thm-ishi-1}
	\bb{E}[\mathcal{L}_k] \leq \frac{1}{12L}\sum_{i = 0}^{k-1}\bb{E}\norm{G(z^i)}^2 + \frac{4}{L}\sum_{i = 1}^{k-1}(i+1)^2(\sigma_i + \gamma_i)^2,\quad k\geq 0.	
	\end{equation}
	Here, the case when $k= 0$ holds trivially since $\mathcal{L}_0 = 0$. By the definition of the potential function $\mathcal{L}_k$ 
	we see that 
	\begin{equation}
	\label{eq-thm-ishi-2}
		\bb{E}[\mathcal{L}_k] = \frac{k(k+1)}{2L}\bb{E}\left[\norm{G(z^k)}^2\right] - (k+1)\bb{E}\left[\inprod{G(z^k)}{z^0 - z^k}\right],\quad k\geq 0.	
	\end{equation}
	Moreover, let $z^*$ be any solution such that $G(z^*) = 0$, since 
	\begin{align*}
	    \bb{E}\left[\inprod{G(z^k)}{z^0 - z^k}\right] = &\; \bb{E}\left[\inprod{G(z^k)}{z^0 - z^*}\right] + \bb{E}\left[\inprod{G(z^k)}{z^* - z^k}\right] \\
        \leq &\;  \bb{E}\left[\inprod{G(z^k)}{z^0 - z^*}\right],
	\end{align*}
	it holds that 
	\begin{equation}
	\label{eq-thm-ishi-3}
		\bb{E}\left[\inprod{G(z^k)}{z^0 - z^k}\right] \leq \bb{E}\left[\norm{G(z^k)}\norm{z^0 - z^*}\right] = \norm{z^0 - z^*}\bb{E}\left[\norm{G(z^k)}\right].	
	\end{equation}
	Combining \eqref{eq-thm-ishi-1}, \eqref{eq-thm-ishi-2} and \eqref{eq-thm-ishi-3}, we get
	\begin{equation*}
		\begin{aligned}
			&\; \bb{E}\left[\norm{G(z^k)}^2 \right] \\
			\leq &\; \frac{1}{6k(k+1)}\sum_{i = 0}^{k-1}\bb{E}\left[\norm{G(z^i)}^2\right] + \frac{2L}{k}\norm{z^0 - z^*}\bb{E}\left[\norm{G(z^k)}\right] + \frac{8}{k(k+1)}\sum_{i = 0}^{k-1}(i+1)^2(\sigma_i + \gamma_i)^2 \\
			\leq &\; \frac{1}{6k(k+1)}\sum_{i = 0}^{k-1}\bb{E}\left[\norm{G(z^i)}^2\right] + \frac{2L}{k}\norm{z^0 - z^*}\left(\bb{E}\left[\norm{G(z^k)}^2\right]\right)^{\frac{1}{2}} \\
   &\; + \frac{8}{k(k+1)}\sum_{i = 0}^{k-1}(i+1)^2(\sigma_i + \gamma_i)^2,
		\end{aligned}
	\end{equation*}
	for all $ k\geq 1 $. Now denote 
	\[
		a_k:= \left(\bb{E}\left[\norm{G(z^k)}^2\right]\right)^{\frac{1}{2}}, \quad k\geq 0, \quad b_k:= \sum_{i = 0}^{k-1}(i+1)^2(\sigma_i + \gamma_i)^2,\quad k\geq 1,\quad D:= L\norm{z^0 - z^*}.
	\]
	Similar to \eqref{thm-rate-proof-2}, we can derive 
	\[
		a_k^2 \leq \frac{1}{3k(k+1)}\sum_{i = 0}^{k-1} a_i^2 + \frac{48}{(k+1)(k+2)}b_k + \frac{24D^2}{(k+1)(k+2)}.
	\]
	Then, by following the proof of Theorem \ref{thm-convergence-rate} directly, we see that 
	\[
		a_k^2\leq \frac{(7D + 10\sqrt{b_k})^2}{(k+1)(k+2)}, \quad k\geq 1,
	\]
	which proves \eqref{eq-thm-rate-ishi-G} since the case for $k = 0$ is trivial. Moreover, we can also verify that
		\begin{align*}
		    \bb{E}\left[\norm{\tilde{z}^k}^2\right] \leq &\;  2\bb{E}\left[\norm{G(z^k)}^2 + \norm{\tilde{z}^k - G(z^k)}^2\right] \\
      \leq &\; 2a_k^2 + 2(\sigma_k + \gamma_k)^2 \\
      \leq  &\; \frac{2(7D + 10\sqrt{b_k})^2}{(k+1)(k+2)} + 2(\sigma_k + \gamma_k)^2.
		\end{align*}
	Then, by following the proof of Theorem \ref{thm-convergence-rate} again, we can show that 
	\[
		a_k^2\leq \frac{(7D + 10\sqrt{b_k})^2}{(k+1)(k+2)},\quad \bb{E}\left[\norm{z^{k+1} - z^k}^2\right] \leq \frac{8(7D + 11\sqrt{b_{k+1}})^2 }{L^2(k+1)(k+2)},\quad k\geq 0,
	\]
	which is exactly \eqref{eq-thm-rate-ishi-Z}. Hence, the proof is completed.
\end{proof} 

\subsection{Proof of Lemma \ref{lemma-PAGE}} \label{proof-lemma-PAGE}

\begin{proof}
	For any randomly sampled index $i\in \{1,\dots, N\}$ and any $z',z\in \bb{Z}$, we can check that 
	\[
		\bb{E}\left[\norm{F_i(z') - F_i(z)}^2\right] \leq L_0^2 \norm{z' - z}^2,
	\]
	due to the $\frac{1}{L_0}$-co-coerciveness of $F_i$ for any $i = 1,\dots, N$. Then, a direct application of \cite[Lemma 2.1]{cai2022stochastic} gives \eqref{eq-page-1}.
	
	We next prove \eqref{eq-page-2} by mathematical induction. For the case $k = 0$, since $S_k^{(1)}$ is the set of i.i.d. samples, it holds that  
		\begin{align*}
		 &	\bb{E}\left[\norm{\widetilde{F}(z^{0}) - F(z^{0})}^2\right] = \; \bb{E}\left[\norm{\frac{1}{N^{(1)}_0}\sum_{i\in {S^{(1)}_0}}F_i(z^0)- F(z^{0})}^2\right]\\
			 = &\;  \bb{E}\left[\frac{1}{\left(N^{(1)}_0\right)^2}\sum_{i\in {S^{(1)}_0}}\norm{F_i(z^0)- F(z^{0})}^2\right] \\
			 \leq &\; \frac{1}{N^{(1)}_0}\sigma^2 \leq \frac{\epsilon^2}{2}< \sigma_0^2.
		\end{align*}
	Hence, \eqref{eq-page-2} holds for the base case. Now assume that \eqref{eq-page-2} hold for all $i < k$, then by \eqref{eq-page-1}, we have that 
		\begin{align*}
			\bb{E}\left[\norm{\widetilde{F}(z^{k}) - F(z^{k})}^2\right] \leq &\; \frac{p_k\epsilon^2}{2(k+1)^{2a}} + (1-p_k)\left( \frac{\epsilon^2}{k^{2a}} + \frac{\epsilon^2}{2(k+1)^{2a+1}}\right) \\
			 \leq &\; \frac{\epsilon^2}{(k+1)^{2a}}\left(\frac{p_k}{2} + \frac{(1-p_k)(k+1)^{2a}}{k^{2a}} + \frac{1-p_k}{2(k+1)} \right) \\
			 = &\; \frac{\epsilon^2}{(k+1)^{2a}}\left((1-p_k) \left(\frac{(k+1)^{2a}}{k^{2a}} + \frac{1}{2(k+1)} - \frac{1}{2} \right) + \frac{1}{2} \right) \\
			 = &\; \frac{\epsilon^2}{(k+1)^{2a}}\left((1-p_k) \frac{2(k+1)^{2a+1} - k^{2a+1}}{2k^{2a}(k+1)} + \frac{1}{2} \right) \\
			 = &\; \frac{\epsilon^2}{(k+1)^{2a}}\left((1-p_k)\frac{1}{2(1-p_k)} + \frac{1}{2} \right) \\
			 = &\; \frac{\epsilon^2}{(k+1)^{2a}}.
		\end{align*}
	Hence, by induction, we have shown that \eqref{eq-page-2} holds for all $k\geq 0$. This completes the proof.
\end{proof}

\subsection{Proof of Corollary \ref{corollary-epsilon-E}} \label{proof-cor-E}
\begin{proof}
	The convergence rates follow immediately from Theorem \ref{thm-rate-ishi}. Then, by the presented rates, we can get the $O(k^{-1})$ iteration complexity after some simple manipulations. Hence, we only need to verify the sample complexity. To this end, by Lemma \ref{lemma-PAGE}, we see that for $a = \frac{1}{2}$, and for $k\geq 1$, 
		\begin{align*}
			p_k = &\; 1 - \frac{\left(\frac{k}{k+1}\right)^{2a}}{2 - \left(\frac{k}{k+1}\right)^{2a+1}} = \frac{3k+2}{k^2+4k+2} < \frac{3}{k+1},\quad 1-p_k = \frac{k^2+k}{k^2+4k+2} < \frac{2k^2}{(k+1)^2},\\
			N_k^{(1)} =  &\; \left\lceil\frac{2\sigma^2}{\epsilon^2(k+1)^{-2a}} \right\rceil \leq   \frac{2(k+1)\sigma^2}{\epsilon^2} + 1,\\
			N^{(2)}_k =  &\;  \left\lceil\frac{2L_0^2\norm{z^k - z^{k-1}}^2}{\epsilon^2(k+1)^{-(2a+1)}}\right\rceil \leq 
			 \frac{2(k+1)^2L_0^2\norm{z^k - z^{k-1}}^2}{\epsilon^2} +1.
		\end{align*}
	Let $N_k$ be the sample size at the $k$-th iteration, we see that 
		\begin{align*}
		    \bb{E}[N_{k}] =&\; p_k\bb{E}\left[N_k^{(1)}\right] + (1-p_k)\bb{E}\left[N_k^{(2)}\right]\\
			\leq &\; \frac{6\sigma^2}{\epsilon^2} + \frac{4k^2L_0^2}{\epsilon^2}\bb{E}\left[\norm{z^k - z^{k-1}}^2\right] + 2 \\
         \leq &\;  O\left(\epsilon^{-2}\right) + O(k^2).
		\end{align*} 
	Note also that $N_0 = N_0^{(1)} = O(\epsilon^{-2})$. Therefore, it holds that 
	\[
		\bb{E}\left[\sum_{k = 0}^KN_k\right] \leq O(K\epsilon^{-2}) + O(K^3), \quad K\geq 0.
	\]  	
	Let $K = O(\epsilon^{-1})$, we then see that the expected total number of stochastic queries to $F$ to get an expected $O(\epsilon)$-optimal solution is $ O\left(\epsilon^{-3}\right) $.  Thus, the proof is completed.
\end{proof}

\subsection{Proof of Corollary \ref{corollary-large-E}} \label{proof-cor-large-E}

\begin{proof}
	We only verify the result of sample complexity. To this end, by Lemma \ref{lemma-PAGE}, we see that $\epsilon = 1$, and for $k\geq 1$, 
		\begin{align*}
			p_k = &\; 1 - \frac{\left(\frac{k}{k+1}\right)^{2a}}{2 - \left(\frac{k}{k+1}\right)^{2a+1}} = \frac{2(k+1)^{2a+1} - (2k+1)k^{2a}}{2(k+1)^{2a+1} - k^{2a+1}} \leq \frac{(k+1)^{2a+1} - (2k+1)k^{2a}}{(k+1)^{2a+1}},\\
			N_k^{(1)} =  &\; \left\lceil\frac{2\sigma^2}{\epsilon^2(k+1)^{-2a}} \right\rceil \leq   2(k+1)^{2a}\sigma^2 + 1,\\
			N^{(2)}_k =  &\;  \left\lceil\frac{2L_0^2\norm{z^k - z^{k-1}}^2}{\epsilon^2(k+1)^{-(2a+1)}}\right\rceil \leq 
			 2(k+1)^{2a+1}L_0^2\norm{z^k - z^{k-1}}^2 +1.
		\end{align*}
	Moreover, notice that since $a>\frac{3}{2}$, $\left(1+\frac{1}{k}\right)^{2a} \leq 1+ \frac{1}{k}\left(4^{a} - 1\right)$ for any $k\geq 1$. Then, we deduce that
		\begin{align*}
			\frac{2(k+1)^{2a+1} - (2k+1)k^{2a}}{k^{2a}} = &\; 2(k+1)\left(1+\frac{1}{k}\right)^{2a} - (2k+1)\\
			\leq &\; 2(k+1)\left(1 + \left(4^a -1\right)\frac{1}{k}\right) - (2k+1) \\
			= &\; 1 + 2\left(4^a-1\right)\left(1+\frac{1}{k}\right) \\
			= &\; O(1).
		\end{align*}
	As a consequence, we see that $p_k = O\left(\frac{k^{2a}}{(k+1)^{2a+1}}\right)$ for all $k\geq 1$.	Let $N_k$ be the sample size at the $k$-th iteration, we see that 
		\begin{align*}
		    \bb{E}[N_{k}] =&\; p_k\bb{E}\left[N_k^{(1)}\right] + (1-p_k)\bb{E}\left[N_k^{(2)}\right] \\
			\leq &\;   O\left(k^{2a-1}\right) + O\left((k+1)^{2a+1}\bb{E}\left[\norm{z^k - z^{k-1}}^2\right]\right) \\
            \leq &\; O\left(k^{2a-1}\right).
		\end{align*}
	Note that $N_0 = N_0^{(1)} = O(1)$. Then, we see that
	\[
		\bb{E}\left[\sum_{k = 0}^KN_k\right] \leq O\left(K^{2a}\right).
	\]
	Let $K = O\left(\delta^{-1}\right)$, we see that the expected total number of stochastic queries to $F$ to get an expected $O(\delta)$-optimal solution is at most $ O\left(\delta^{-2a}\right) $. This completes the proof.
\end{proof}

\end{document}